\documentclass[10pt,a4paper]{article}
\usepackage[utf8]{inputenc}
\usepackage{amsmath}
\usepackage{amsfonts}
\usepackage{amssymb}
\usepackage{amsthm}
\usepackage[dvipsnames,table]{xcolor}
\usepackage{enumerate}
\usepackage{graphicx}
\usepackage{tikz}
\usepackage{hyperref}
\usepackage[normalem]{ulem}
\usepackage[misc]{ifsym}
\usepackage{multicol}
\usepackage{nomencl}
\usepackage{comment}

\usepackage[margin=2.5cm]{geometry}

\makenomenclature

\renewcommand*{\nompreamble}{\begin{multicols}{2}}
\renewcommand*{\nompostamble}{\end{multicols}}
 \newtheorem{thm}{Theorem}[section]
 \newtheorem{cor}[thm]{Corollary}
 \newtheorem{lem}[thm]{Lemma}
 \newtheorem{prop}[thm]{Proposition}
 \newtheorem{thmdef}[thm]{Theorem/Definition}
 \theoremstyle{definition}
 \newtheorem{defn}[thm]{Definition}
 \newtheorem{assumption}{Assumption}
 \theoremstyle{remark}
 \newtheorem{rem}[thm]{Remark}

 \newtheorem{ex}{Example}
 \newtheorem*{exs}{Examples}
 \numberwithin{equation}{section}
 
\newcommand{\vertiii}[1]{{\left\vert\kern-0.25ex\left\vert\kern-0.25ex\left\vert #1
  \right\vert\kern-0.25ex\right\vert\kern-0.25ex\right\vert}}

\newcommand{\tr}{\operatorname{Tr}}

\newcommand{\BUC}{\operatorname{BUC}(\Xi)}

\hypersetup{verbose=false, colorlinks, allcolors=MidnightBlue, linkcolor=MidnightBlue, urlcolor=OliveGreen}

\title{Wiener's Tauberian theorem in classical and quantum harmonic analysis}
\author{Robert Fulsche, Franz Luef and Reinhard F.\ Werner}

\begin{document}

\maketitle
\abstract{We investigate Wiener's Tauberian theorem from the perspective of \emph{limit functions}, which results in several new versions of the Tauberian theorem. Based on this, we formulate and prove analogous Tauberian theorems for operators in the sense of \emph{quantum harmonic analysis}. Using these results, we characterize the class of \emph{slowly oscillating operators} and show that this class is strictly larger than the class of uniformly continuous operators. Finally, we discuss uniform versions of Wiener's Tauberian theorem and its operator analogue and provide an application of this in operator theory.

\vspace{0.25cm}
}

\section{Introduction}

The framework of \emph{quantum harmonic analysis}, as introduced in the paper \cite{werner84}, has received significant attention in the last few years, leading to numerous important results and applications in time-frequency analysis, operator theory and theoretical physics, see, e.g., \cite{Berge_Berge_Luef_Skrettingland2022, Dammeier_Werner2023, Fulsche_Hagger2023, Fulsche_Rodriguez2023, Halvdansson2022, keyl_kiukas_werner16, Luef_Skrettingland2018a, Luef_Skrettingland2019a}. In analogy to classical harmonic analysis, at the heart of quantum harmonic analysis lies a notion of convolution between two operators and/or functions, extending the classical notion of the convolution of two functions. Building on this, several results analogous to statements from classical harmonic analysis are obtained, for example, Wiener's approximation theorems \cite{werner84} or results on spectral synthesis \cite{Fulsche_Rodriguez2023}.

In the recent work \cite{Luef_Skrettingland2021}, a result in the realm of quantum harmonic analysis was provided, which is analogous to Wiener's Tauberian theorem in classical harmonic analysis (cf.\ Theorem 5.1 there). We recall that Wiener's classical Tauberian theorem states that, for $f \in L^\infty(\mathbb R)$, we have that
\begin{align*}
    g \ast f \in C_0(\mathbb R)
\end{align*}
for all $g \in L^1(\mathbb R)$ if and only if this holds for just one regular $g \in L^1(\mathbb R)$ (meaning that this $g$ has nowhere vanishing Fourier transform). While this equivalence is clearly related to certain properties of $f$, it is seemingly not possible to deduce any properties of $f$ if the above equivalent statements hold. Nevertheless, deducing properties of $f$ from properties of $g \ast f$ is possible if one assumes that $f$ is a slowly oscillating function, and this is the content of Pitt's refinement (which may be regarded as a converse to  Wiener's Tauberian theorem): If $f \in L^\infty(\mathbb R)$ is slowly oscillating, and $g \ast f \in C_0(\mathbb R)$ for all $g \in L^1(\mathbb R)$ (equivalently, for one regular $g$), then $f$ vanishes at infinity. Hence, under the assumption of slow oscillation, properties of the convolutions $g \ast f$ enforce properties of $f$.

Returning now to the setting of quantum harmonic analysis, F.~Luef and E.~Skrettingland observed in \cite[Theorem 5.2]{Luef_Skrettingland2021} that the following statement from \emph{correspondence theory} may be interpreted as some kind of Pitt's refinement for the Wiener Tauberian theorem for operators: If $B \in \mathcal L(L^2(\mathbb R))$ satisfies a certain regularity (it is contained in the algebra $\mathcal C_1$) and $A \ast B \in C_0(\mathbb R^2)$ for some regular trace class operator $A \in \mathcal T^1(L^2(\mathbb R))$, then $B$ is compact. For a precise definition of the convolution $A \ast B$ and regular operators, we refer to Section \ref{sec:quantum}.

Indeed, the algebra $\mathcal C_1$ mentioned above takes the same role as the algebra of bounded, uniformly continuous functions does in classical harmonic analysis. Since the class of bounded, slowly oscillating functions strictly contains the class of bounded, uniformly continuous functions, the following question is quite natural: Is the above-mentioned version of Pitt's refinement for operators in some sense ``optimal", or in other words is there a class of operators larger than $\mathcal C_1$ which permits a Pitt-type statement? This question can be rephrased as ``What is the correct class of \emph{slowly oscillating operators}?'' and was indeed one of the main motivations for our investigations.

In order to gain insight into how to approach this set of ideas, we first revisit Wiener's classical Tauberian theorem from a rather general point of view, using the theory of \emph{limit functions}: For $f \in L^\infty(\mathbb R)$, limit functions are, roughly speaking, those functions on $\mathbb R$ that one can obtain as limits of $\alpha_x(f) = f(\cdot - x)$ for $x \to \pm \infty$ (in a certain topology). By doing so, we obtain two new results (Theorems \ref{mainthm:1} and \ref{mainthm:2} in this paper) that extend the classical theorem. While Wiener's Tauberian theorem is by now almost 100 years old (it seems to have first appeared in \cite{Wiener1932}) and has been the subject of intense and detailed investigations, we were unable to find any results similar to ours in the literature. For example, the impressive monograph \cite{Korevaar2004} on Tauberian theory does not mention the theory of limit functions at all.

These generalizations of the classical results give the correct perspective on the operator theory, which first of all provides analogous versions of the before-mentioned generalizations to the setting of quantum harmonic analysis. Having obtained these results, we are then able to describe the right class of slowly oscillating operators in terms of their behaviour at infinity, utilizing the language of \emph{limit operators}. Similarly to limit functions, limit operators are defined, in the simplest case, for $A \in \mathcal L(L^2(\mathbb R))$, as limits of $\alpha_{(x, \xi)}(A)$ for $(x, \xi) \to \infty$, where $\alpha_{(x, \xi)}(A) = U_{(x, \xi)} A U_{(x, \xi)}^\ast$ is the action of $\mathbb R^2$ on linear operators by adjoining with the time-frequency shifts $U_{(x, \xi)}$.   

At the end of the paper, we discuss uniform versions of Wiener's Tauberian theorem, which are not closely related to the results in the previous sections, but are part of the problems concerning Wiener's Tauberian theorem and quantum harmonic analysis. More precisely, in \cite[Prop.\ 4.7]{Luef_Skrettingland2021}, it is shown that for $f \in L^\infty(\mathbb R^{2d})$, it holds true that $g \ast f \in C_0(\mathbb R^{2d})$ for some regular $g \in L^1(\mathbb R^{2n})$ if and only if there exists some $0\neq \Phi \in \mathcal S(\mathbb R^{2d})$ such that
\begin{align*}
    \lim_{|x| \to \infty} \sup_{|\xi| \leq R} |V_\Phi f (x, \xi)| = 0
\end{align*}
for every $R > 0$. Here, $V_\Phi f$ is the \emph{short time Fourier transform} of $f$ with window $\Phi$. Even though the statement looks like a result of classical harmonic analysis, the proof relied on methods of quantum harmonic analysis. As a result, the statement could only be proven for functions on $\mathbb R^n$ for \emph{even} $n \in \mathbb N$. In this section, we will provide a proof of this statement on arbitrary locally compact abelian groups which works directly within the framework of classical harmonic analysis, i.e., through a uniform form of Wiener's Tauberian theorem. We also provide analogous uniform theorems for quantum harmonic analysis. As an application, we give a novel compactness characterization for bounded linear operators.

The paper is organized in three parts: In Section \ref{sec:classical}, we discuss the classical part of the problem. We give a somewhat more detailed introduction to Wiener's classical Tauberian theorem and then discuss and prove our results on the classical side. For appropriate generality and later reference, we decided to present the discussion on arbitrary locally compact abelian groups which are non-compact (as the whole theory of Wiener's Tauberian theorem is rather trivial on compact groups). In Section \ref{sec:quantum}, we then extend our results from the classical level to the setting of quantum harmonic analysis. Again, we strive for some generality, presenting the results for quantum harmonic analysis on general abelian phase spaces as discussed in \cite{Fulsche_Galke2023}.  Finally, in Section \ref{sec:uniform}, we present the results on uniform Wiener Tauberian theorems both in classical and quantum harmonic analysis. Since our work contains extensive notation, we provide a short list of notation used at the end.

\section{Wiener's classical Tauberian theorem and limit functions}\label{sec:classical}
Standard references for the following facts on harmonic analysis of locally compact abelian groups are \cite{hewitt_ross_2, reiter20}. Let $G$ be a locally compact abelian group (in the following abbreviated as \emph{lca group}). 
\begin{rem}
If the reader is not familiar with the theory of locally compact abelian groups, it might still be worth reading through the paper, replacing each occurrence of the locally compact abelian group $G$ by $\mathbb R^n$, which is probably the most important instance of an lca group. Even in this particular setting, our results seem to be new and of interest. In this particular case, the dual group $\widehat{G}$ should be replaced with the set of all functions $f_\xi(x) = e^{ix\cdot \xi}$, where $\xi \in \mathbb R^n$ and $x \cdot \xi$ denotes the usual inner product. The Haar measure is in this case simply the Lebesgue measure.
\end{rem}

We will always denote the group operation of $G$ additively. Note that our definition of an lca group includes the Hausdorff property. By $\widehat{G}$ we denote the dual group of $G$ (i.e.\ the group of continuous homomorphisms $G \to \mathbb T$ endowed with the compact-open topology), which is again an lca group. By $e$ we will usually denote the neutral element of $G$. The measure $dx$ will denote any choice of a Haar measure on $G$. We denote the Fourier transform of $f \in L^1(G)$ by
\begin{align*}
\nomenclature{$\mathcal F f$}{Fourier transform of $f$, p.\  \pageref{def:fouriertrafo}}\mathcal Ff (\chi) = \widehat{f}(\chi) = \int_G \overline{\chi(x)} f(x)~dx.
\end{align*}
\label{def:fouriertrafo}
Then, by Plancherel's theorem, there is a normalization $d\chi$ of the Haar measure on the dual group $\widehat{G}$ such that the Fourier transform extends to a unitary operator from $L^2(G)$ to $L^2(\widehat{G})$, and we fix this normalization of $d\chi$ in the following. In that case, it holds true that $\mathcal F\mathcal F f = f(-\cdot)$ (where $\mathcal F \mathcal F$ denotes the composition of the Fourier transform from $G$ to $\widehat{G}$ and from $\widehat{G}$ to $G$). Recall the classical approximation theorem of Wiener:\setcounter{thm}{0}
\begin{thm}[Wiener's approximation theorem]
Let $S \subset L^1(G)$. Then, the translates of functions from $S$ span a dense subset of $L^1(G)$ if and only if
\begin{align*}
\bigcap_{g \in S} \{ \chi \in \widehat{G}: \widehat{g}(\chi) = 0\} = \emptyset.
\end{align*}
\end{thm}
We will call a subset $S \subset L^1(G)$ satisfying the above property a \emph{regular subset of $L^1(G)$}. A function $g \in L^1(G)$ will be called regular if $\{ g\}$ is a regular subset of $L^1(G)$, i.e.\ if $\widehat{g}$ vanishes nowhere. Wiener's approximation theorem is usually used to derive Wiener's Tauberian theorem (and is indeed equivalent to this result):
\begin{thm}[Wiener's Tauberian theorem]
Let $S \subset L^1(G)$ be a regular subset and $f \in L^\infty(G)$. If there is some $a \in \mathbb C$ with
\begin{align*}
\lim_{x \to \infty} g \ast f(x) = a \int_G g(x)~dx
\end{align*}
for every $g \in S$, then the same equation holds for every $g \in L^1(G)$.
\end{thm}
Clearly, the equality in Wiener's Tauberian theorem can be rephrased as
\begin{align*}
g \ast f - a \int_G g(x)~dx \in C_0(G).
\end{align*}
Here, \nomenclature{$C_0(G)$}{Continuous functions on $g$ vanishing at infinity, p.\ \pageref{def:c0}}\label{def:c0}$C_0(G)$ is the $C^\ast$-algebra of continuous functions vanishing at infinity. Recall that a continuous function $f: G \to \mathbb C$ vanishes at infinity if for each $\varepsilon > 0$ there exists a compact set $K \subset G$ such that $|f(x)| < \varepsilon$ for every $x \in G \setminus K$. We will denote by $C_b(G)$\nomenclature{$C_b(G)$}{bounded, continuous functions on $G$, p.\ \pageref{def:cb}}\label{def:cb} the bounded continuous functions on $G$. Further, $\operatorname{BUC}(G)$\nomenclature{$\operatorname{BUC}(G)$}{bounded, uniformly continuous functions on $G$, p.\ \pageref{def:buc}}\label{def:buc} denotes the bounded and uniformly continuous functions on $G$. More concretely, this can be defined for any lca group $G$ as
\begin{align*}
\operatorname{BUC}(G) = \{ f \in C_b(G): ~&\forall \varepsilon > 0~~ \exists \text{ a neighborhood }\, O \text{ of } e: \\
&\forall\, x \in G,\, y \in O:\, |f(x) - f(x-y)| < \varepsilon\}.
\end{align*}

For a function $f: G \to \mathbb C$ and some point $x \in G$ we will usually write
\begin{align*}
\nomenclature{$\alpha_x(f)$}{Shift of function, p.\ \pageref{def:shiftoffunction}, or limit function, p\ \pageref{thm:deflimitfcts}} \alpha_x(f)(y) = f(y-x), \quad \nomenclature{$\beta_-(f)$}{Parity of function, p.\ \pageref{def:shiftoffunction}}\beta_-(f)(y) = f(-y).\nonumber
\end{align*}
\label{def:shiftoffunction} A space $X$, consisting of functions on $G$, will be called \emph{translation-invariant} if we have $\alpha_x(f) \in X$ for every $x \in G$ and $f \in X$. 

In the above version of Wiener's Tauberian theorem, $C_0(G)$ can be replaced by any closed, translation-invariant subspace $\mathcal D$ of $\operatorname{BUC}(G)$. By essentially the same proof as for the classical result, we obtain:
\begin{thm}[Wiener's Tauberian theorem, version 2]
Let $\mathcal D \subset \operatorname{BUC}(G)$ be a closed, translation-invariant subspace. Further, let $S \subset L^1(G)$ be a regular subset and $f \in L^\infty(G)$. If there is some $a \in \mathbb C$ with
\begin{align*}
g \ast f - a \int_G g(x)~dx \in \mathcal D
\end{align*}
for every $g \in S$, then this membership in $\mathcal D$ is true for every $g \in L^1(G)$.
\end{thm}
Stated like this, Wiener's Tauberian theorem is a statement about the membership of many convolutions in a translation-invariant closed subspace of $\operatorname{BUC}(G)$. Note that when $1 \in \mathcal D$, there is no loss of generality in setting $a = 0$ in the theorem, i.e.\ $g \ast f \in \mathcal D$ if and only if $g \ast f - a\int_G g(x)~dx \in \mathcal D$, hence we do not need to care about the part $a\int_G g(x)~dx$.

Frequently, one will also encounter Pitt's extension of Wiener's Tauberian theorem:
\begin{thm}[Pitt's extension to Wiener's Tauberian theorem]
Let $f \in \operatorname{SO}(G)$ and $g \in L^1(G)$ be regular. Then, $f \in B_0(G)$ if and only if $g \ast f \in C_0(G)$. 
\end{thm}
We will define the spaces $\operatorname{SO}(G)$ and $B_0(G)$ later in detail, they stand for \emph{slowly oscillating functions} and \emph{functions vanishing at infinity}. Finally, Wiener's approximation theorem has another consequence, which does not bear a particular name, but is not hard to prove (using the existence of a bounded approximate unit in $L^1(G)$):
\begin{thm}\label{thm:bucinv}
Let $\mathcal D$ be a closed and translation-invariant subspace of $\operatorname{BUC}(G)$. Further, let $S \subset L^1(G)$ be a regular subset. If $f \in \operatorname{BUC}(G)$, then:
\begin{align*}
f \in \mathcal D \Leftrightarrow g \ast f \in \mathcal D \text{ for every } g \in S \Leftrightarrow g \ast f \in \mathcal D \text{ for every } g \in L^1(G).
\end{align*}
\end{thm}
All these results have some common features: They are statements relating regular subsets of $L^1(G)$ and membership of convolutions in closed, translation-invariant subspaces of $\operatorname{BUC}(G)$. There are nevertheless some imminent open questions related to this problem, for example: Can one characterize
\begin{align*}
\{ f \in L^\infty(G): ~g \ast f \in \mathcal D \text{ for } g \in L^1(G)\}
\end{align*}
for some given closed, translation-invariant subspace of $\operatorname{BUC}(G)$ directly, i.e.\ without using convolutions by a regular subset of $L^1(G)$? And what is the precise connection between the three spaces $SO(G)$, $B_0(G)$ and $C_0(G)$ occurring in Pitt's extension? In this part of the paper, we will present an answer to these questions for classical harmonic analysis by the theory of limit functions.

\subsection{$L^\infty(G)$, $\operatorname{BUC}(G)$ and shifts to infinity}
If $G$ is a second-countable locally compact abelian group, then it is well-known that the space $L^\infty(G)$ consisting of (equivalence classes of) essentially bounded Borel measurable functions can be identified with the dual space of $L^1(G)$ (with respect to some choice of Haar measure) by:
\begin{align*}
\Psi_f(g) = \int_G g(x) f(x)~dx, \quad g \in L^1(G), ~f \in L^\infty(G).
\end{align*}
If $G$ is not second countable and hence the Haar measure is no longer $\sigma$-finite, this identification of the dual of $L^1(G)$ with $L^\infty(G)$ breaks down. Nevertheless, it turns out that the failure of this duality is simply due to the fact that the space $L^\infty(G)$ was defined in the wrong way: In general, $L^\infty(G)$ has to be interpreted as equivalence classes of functions which are locally-Borel-measurable and which are bounded outside of locally-null sets. For details, we refer to \cite[Section 2.3]{Folland2016}. For the reader's convenience, we illustrate these concepts in an example:
\begin{ex}
    We consider the lca group $G = \mathbb R \times \mathbb R_d$, where $\mathbb R_d$ denotes the group of real numbers with the discrete topology. Then, the subset $\{0\} \times \mathbb R_d \subset \mathbb R \times \mathbb R_d$ is a Borel set with infinite Haar measure. Nevertheless, whenever this set has a non-trivial intersection with a compact subset of $\mathbb R \times \mathbb R_d$, the result is a set of Haar measure zero (this is a simple consequence of the fact that for each compact $K \subset \mathbb R \times \mathbb R_d$ we necessarily have that $\{ y: ~(x,y) \in K \text{ for some } x \in \mathbb R\}$ needs to be finite). Hence, $\{ 0\} \times \mathbb R_d$ is a Borel set with non-zero Haar measure but satisfies the locally-null condition that we did not precisely introduce here. The indicator function $\mathbf 1_{\{ 0 \} \times \mathbb R_d}$ is clearly Borel measurable and not zero with respect to the equivalence relation of equality outside null sets, but it is zero with respect to the equivalence relation of equality outside locally-null sets. It is not hard to verify that $\int_{\{ 0\} \times \mathbb R_d} g(x)~dx = 0$ for every $g \in L^1(G)$, so this function corresponds to the zero functional on $L^1(G)$.
\end{ex}

With these conventions, $L^\infty(G)$ indeed turns out to be the dual space of $L^1(G)$ in a natural way and isometrically, and we will always understand $L^\infty(G)$ in this way. Further, we will usually write $\langle f, g\rangle = \int_G f(x) g(x)~dx$ for the dual pairing between $g \in L^1(G)$ and $f \in L^\infty(G)$.

From now on, we tacitly assume that $G$ is non-compact. While essentially everything is still true in the compact case, the statements are rather trivial and not particularly interesting.

The following fact is well-known and important to our presentation:
\begin{prop}
Let $G$ a lca group. Then, 
\begin{align*}
\operatorname{BUC}(G) = \{ f \in L^\infty(G): x \mapsto \alpha_x(f) \text{ is continuous w.r.t. } \| \cdot\|_\infty\}.
\end{align*}
\end{prop}
\begin{proof}[Comment and proof]
Note that there is some ambiguity here: Above, elements from $\operatorname{BUC}(G)$ were, by definition, continuous functions. In the proposition, we wrote them as elements in $L^\infty(G)$ on which the shifts act strongly continuous. It is a priori not clear if such elements even have continuous representatives. Indeed, such a representative always exists, and its construction also proves the proposition: We recall that an \emph{approximate identity of $L^1(G)$} is a net of functions $(\varphi_\gamma)_{\gamma \in \Gamma} \subset L^1(G)$ such that $f \ast \varphi_\gamma \overset{\gamma \in \Gamma}{\longrightarrow} f$ in $L^1$-norm for every $f \in L^1(G)$. Such an approximate identify always exists and can be chosen normalized (i.e., with $\| \varphi_\gamma\|_{L^1} = 1$ for each $\gamma)$ and positive; for example, given a neighborhood base $\mathcal O$ (ordered decreasingly by inclusion) of $e \in G$, the functions $\psi_{O} = \frac{\chi_{O}}{|O|}$, $O \in \mathcal O$, form an approximate identity.

We first show that each element of the set on the right-hand side has a continuous representative, which is contained in $\BUC(G)$. If we let $\varphi_\gamma$ be any non-negative approximate identity in $L^1(G)$ with $\beta_-(\varphi_\gamma) = \varphi_\gamma$ and $f \in L^\infty(G)$ such that the shifts act continuously on $f$, then $\varphi_\gamma \ast f$ is defined pointwise as
\begin{align*}
\varphi_\gamma \ast f(x) &= \int_{G} \varphi_\gamma(x-y) f(y)~dy = \int_G \varphi_\gamma(y-x)f(y)~dy \\
&= \langle \alpha_x(\varphi_\gamma), f\rangle = \langle \varphi_\gamma, \alpha_{-x}(f)\rangle.
\end{align*}
Using this, it is not hard to verify that $\varphi_\gamma \ast f$ is a bounded uniformly continuous function in the sense of a function and not as an equivalence class. For continuous functions, the essential supremum equals the supremum, hence $\varphi_\gamma \ast f$ is a Cauchy net in the supremum norm and thus it converges to some uniformly continuous function. Further, $\varphi_\gamma \ast f \to f$ in the sense of equivalence classes. Hence, we obtained a bounded uniformly continuous representative of the equivalence class $f$. In particular, we also have shown the inclusion ``$\supseteq$'' of the statement. 

The inclusion ``$\subseteq$'' is a straightforward consequence of the definition of the space $\BUC(G)$, together with the triangle inequality.
\end{proof}

Clearly, $\operatorname{BUC}(G)$ is a unital $C^\ast$-subalgebra of $C_b(G)$. In particular, $\mathcal M(\operatorname{BUC}(G))$, the maximal ideal space, can be considered as a compactification of $G$, identifying points $x \in G$ with the point evaluation functionals $\delta_x$. This compactification is sometimes also called the \emph{Samuel compactification of $G$}, which we will usually abbreviate as \nomenclature{$\sigma G$}{Samuel compactification of $G$, p.\ \pageref{def:Scomp}}\label{def:Scomp}$\sigma G := \mathcal M(\operatorname{BUC}(G))$. We will also write \nomenclature{$\partial G$}{Boundary of $G$ in its Samuel compactification, p.\ \pageref{def:boundG}}\label{def:boundG}$\partial G:= \sigma G \setminus G$ for the boundary of $G$ in its Samuel compactification.

We denote by \nomenclature{$C_c(G)$}{Continuous, compactly supported functions on $G$, p.\ \pageref{def:Cc}}\label{def:Cc}$C_c(G)$ the set of continuous functions with compact support. For $g \in C_c(G)$, $x \mapsto \alpha_x(g)$ is of course continuous in the uniform topology. This easily implies that $x \mapsto \alpha_x(g)$ is also continuous in $L^1(G)$-norm with $\| \alpha_x(g)\|_{L^1(G)} = \| g\|_{L^1(G))}$. Since $C_c(G)$ is dense in $L^1(G)$, $L^1$-isometry and $L^1$-continuity of the shift holds true for all $g \in L^1(G)$:
\begin{align*}
\text{For } g \in L^1(G),~x \mapsto \alpha_x(g) \text{ is uniformly continuous w.r.t. } \| \cdot\|_{L^1(G)}.
\end{align*}
By duality, it is clear that for each $f \in L^\infty(G)$ the map $G \ni x \mapsto \alpha_x(f)$ is continuous with respect to the weak$^\ast$ topology. This fact has the following important consequence:
\begin{thmdef}\label{thm:deflimitfcts}
Let $f \in L^\infty(G)$. Then, the continuous map
\begin{align*}
    G \ni x \mapsto \alpha_x(f) \in L^\infty(G),
\end{align*}
considered with respect to the weak$^\ast$ topology, extends to a continuous map from $\sigma G$ to $L^\infty(G)$ (again, with the weak$^\ast$ topology). For $x \in \partial G$, we still denote the value of this map by $\alpha_x(f)$.
\end{thmdef}
\begin{proof}
Since $L^\infty(G)$ is the dual space of $L^1(G)$, the function
\begin{align*}
G \ni x \mapsto \langle g, \alpha_x(f)\rangle = \langle \alpha_{-x}(g), f\rangle
\end{align*}
is bounded and uniformly continuous for every pair $(g, f) \in L^1(G) \times L^\infty(G)$. In particular, this can be extended to a continuous function
\begin{align*}
\sigma G \ni x \mapsto \varphi_{(g, f)}(x)
\end{align*}
with $\varphi_{(g,f)}(x) = \langle g, \alpha_x(f)\rangle$ for $x \in G$. For $x \in \sigma G$ and $f \in L^\infty(G)$ fixed, $g \mapsto \varphi_{(g, f)}(x)$ is a bounded linear functional of $L^1(G)$, so there is a unique function $\alpha_x(f) \in L^\infty(G)$ such that
\begin{align*}
\varphi_{(g,f)}(x) = \langle g, \alpha_x(f)\rangle.
\end{align*}
This finishes the proof.
\end{proof}
The functions $\alpha_x(f)$, $x \in \partial G$, are called the \emph{limit functions of $f$}. We list some of their elementary properties:
\begin{lem}\label{lem:prop_limit_fcts}
    Let $f \in L^\infty(G)$ and $x \in \sigma G$.
    \begin{enumerate}
        \item The map $f \mapsto \alpha_x(f)$ is linear.
        \item  The norm of the limit functions can be estimated as $\| \alpha_x(f)\|_\infty \leq \| f\|_\infty$.
        \item Let $y \in G$. Then, it holds true that $\alpha_x(\alpha_y(f)) = \alpha_y(\alpha_x(f))$.
    \end{enumerate}
\end{lem}
\begin{proof}
\begin{enumerate}
    \item This follows immediately from the definition.
    \item Since $L^\infty(G)$ can be isometrically identified with the dual space of $L^1(G)$, it holds true that \begin{align*} 
    \| f\|_\infty = \sup_{g\in L^1; ~\| g\| = 1} |\langle g, f\rangle|.
    \end{align*}
    Hence, for a net $(x_\gamma)_{\gamma \in \Gamma} \subset G$ converging to $x \in \sigma G$, we see that (using that $\| \alpha_{-x_\gamma}(g)\|_{L^1} = \| g\|_{L^1}$):
    \begin{align*}
        \sup_{g \in L^1; ~\| g\|_{L^1} = 1} |\langle g, \alpha_{x}(f)\rangle| &= \sup_{g \in L^1; ~\| g\|_{L^1} = 1} \lim_{\gamma \in \Gamma} |\langle \alpha_{-x_\gamma}(g), f\rangle|\\
        &\leq \sup_{g \in L^1; ~\| g\|_{L^1} = 1} |\langle g, f\rangle| = \| f\|_\infty.
    \end{align*}
    \item For $x \in G$ this is clear. For $x \in \partial G$, this follows right away from the definition of $\alpha_x(f)$: If $f \in L^\infty(G)$ and $g \in L^1(G)$, then
    \begin{align*}
        \langle g, \alpha_x(\alpha_y(f))\rangle &= \lim_{\gamma \in \Gamma} \langle g, \alpha_{x_\gamma} (\alpha_y(f))\rangle = \lim_{\gamma \in \Gamma} \langle \alpha_{-x_\gamma}(\alpha_{-y}(g)), f\rangle\\
        &= \lim_{\gamma\in \Gamma} \langle \alpha_{-y}(\alpha_{-x_\gamma}(g)), f\rangle = \langle g, \alpha_y(\alpha_x(f))\rangle,
    \end{align*}
    which finishes the proof.
\end{enumerate}
\end{proof}

We will now show that for functions in $\mathrm{BUC}(G)$, the functions $\alpha_x(f)$, where $x\in \partial G$, can also be expressed in a different way. To this end, we will make use of a map $\rho: \sigma G \to \sigma G$. Before properly defining this map, we note that we will often write $\rho x$ instead of $\rho(x)$, to spare the reader an additional pair of brackets in equations that already contain many.

Recall that elements of $\sigma G$ are multiplicative linear functionals on $\mathrm{BUC}(G)$. The multiplicative linear functional $\rho x$ is defined to act as $\rho x(f) = x(\beta_-(f))$\nomenclature{$\rho(x)$}{Dual action of $\beta_-$ on $\sigma G$, p.\ \pageref{def:rho}}\label{def:rho}. When $x \in G$, then this simply boils down to $\rho x = -x$. 

Suppose $f \in \operatorname{BUC}(G)$ and $x \in \sigma G$ and let $(x_\gamma)_{\gamma \in \Gamma} \subset G$ be a net converging to $x$. Then, one has
\begin{align}\label{limfkt:pointwise}
\alpha_{x_\gamma}(f)(y) = f(y-x_\gamma) = \delta_{x_\gamma}\beta_-(\alpha_y(f))\overset{\gamma \in \Gamma}{\longrightarrow} x(\beta_- \alpha_y(f)).
\end{align}
 This shows that $\alpha_{x_\gamma}(f)$ converges pointwise to the function on the right-hand side of \eqref{limfkt:pointwise} (as a function of $y$). Note that, upon identifying a function $f \in \operatorname{BUC}(G)$ with a continuous function on $\sigma G$, we have $x(\beta_- \alpha_y(f)) = f(\rho x)$. As an easy consequence of the Arzel\`{a}-Ascoli theorem (as the shifts $\alpha_{x_\gamma}(f)$ are of course uniformly bounded and uniformly equicontinuous), this convergence holds even uniformly on compact subsets of $G$. From this, it is easy to show that the function written on the right-hand side of Eq.\ \eqref{limfkt:pointwise} agrees with the weak$^\ast$ limit $\alpha_x(f)$ defined in Theorem \ref{thm:deflimitfcts}. Hence, for $f \in \mathrm{BUC}(G)$ it is justified to write
 \begin{align*}
     \alpha_x(f)(y) = x(\beta_- \alpha_y(f)) = \rho x(\alpha_y(f)), \quad x \in \sigma G, ~y \in G,
 \end{align*}
 and the convergence $\alpha_{x_\gamma}(f) \to \alpha_x(f)$ holds uniformly on compact subsets of $G$.
 
 Thus, we have shown the first assertion of the following statement:
\begin{prop}
Let $f \in \operatorname{BUC}(G)$. Then, for any net $(x_\gamma)\subset G$ converging to $x \in \sigma G$, the convergence $\alpha_{x_\gamma}(f) \to \alpha_x(f)$ holds uniformly on compact subsets of $G$; further, $\alpha_x(f) \in \operatorname{BUC}(G)$.
\end{prop}
\begin{proof}
Using the properties listed in Lemma \ref{lem:prop_limit_fcts} we have
\begin{align*}
    \| \alpha_y(\alpha_x(f)) - \alpha_x(f)\|_\infty = \| \alpha_x(\alpha_y(f) - f)\|_\infty \leq \| \alpha_y(f) - f\|_\infty,
\end{align*}
which completes the proof.
\end{proof}

Note that the limit functions behave well with respect to convolutions:
\begin{lem}
Let $g \in L^1(G)$, $f \in L^\infty(G)$ and $x \in \sigma G$. Then, 
\begin{align*}
\alpha_x(g \ast f) = g \ast \alpha_x(f).
\end{align*}
\end{lem}
\begin{proof}
We have for $h \in L^1(G)$ that $\langle h, g \ast f\rangle = \langle h \ast \beta_- g, f\rangle$. Hence, for a net $(x_\gamma) \subset G$ converging to $x \in \sigma G$, it is:
\begin{align*}
\langle h, \alpha_x(g \ast f)\rangle &= \lim_{\gamma} \langle h, \alpha_{x_\gamma}(g \ast f)\rangle = \lim_\gamma \langle h, g \ast \alpha_{x_\gamma}(f)\rangle\\
&= \lim_\gamma \langle h \ast \beta_- g, \alpha_{x_\gamma}(f)\rangle = \langle h \ast \beta_- g, \alpha_x(f)\rangle\\
&= \langle h, g \ast \alpha_x(f)\rangle.
\end{align*}
Since $h \in L^1(G)$ was arbitrary, this shows $\alpha_x(g \ast f) = g \ast \alpha_x(f)$.
\end{proof}
\subsection{The main results for the classical case}\label{subsec:22}
In the following, we fix $\mathcal A$, a closed unital $C^\ast$-subalgebra of $\operatorname{BUC}(G)$ which is invariant under $\alpha_x$ for every $x \in G$. Further, we fix a Hausdorff vector topology $\tau$ on $L^\infty(G)$, which is finer than the weak$^\ast$ topology on $L^\infty(G)$ (not necessarily strictly finer). 
\begin{rem}
In practice, we will only be interested in the cases where $\tau = \tau_{w^\ast}$\nomenclature{$\tau_{w^\ast}$}{weak$^\ast$ topology on $L^\infty(G)$ or on $\mathcal L(\mathcal H)$, p.\ \pageref{def:wstar} and p.\ \pageref{def:wstartopoperators}}\label{def:wstar}, the weak$^\ast$ topology, and $\tau = \tau_{u.c.}$\nomenclature{$\tau_{u.c.}$}{topology of uniform convergence on compact sets on $L^\infty(G)$, p.\ \pageref{def:co}},\label{def:co} the topology of uniform convergence on compact sets. One may keep these two cases in mind when reading through the following.
\end{rem}
For each such $\mathcal A$, we can consider the maximal ideal space $\mathcal M(\mathcal A)$ as a compact space topological space in which $G$ is continuously and densely mapped through the map $G \ni x \mapsto \delta_x \in \mathcal M(\mathcal A)$. This map is injective if and only if $\mathcal A$ separates points of $G$. We will always identify $G$ with its image in $\mathcal M(\mathcal A)$, which will cause no trouble, even if the identification is not injective. Further, we define $\partial_{\mathcal A}(G) := \mathcal M(\mathcal A) \setminus G$. Note that $\partial_{\mathcal A}(G)$ is closed in $\mathcal M(\mathcal A)$ if $C_0(G) \subset \mathcal A$, but in general this is false: For example, when $\mathcal A$ consists of the almost periodic functions, then $\partial_\mathcal{A}(G)$ is dense in $\mathcal M(\mathcal A)$. Based, among others, on these facts, we will assume that $\mathcal A$ contains $C_0(G)$, see also Remark \ref{remark:C0} below on this assumption.

We will make use of the following two topological lemmas:
\begin{lem}
    Let $X$ be a topological space, $A \subset X$ a dense subset, and $Y$ a regular topological space. If $f: A \to Y$ is such that for each $x \in X \setminus A$ and each net $(x_\gamma)_{\gamma \in \Gamma} \subset A$ with $x_\gamma \overset{\gamma \in \Gamma}{\longrightarrow} x$, the limit $\lim_{\gamma \in \Gamma} f(x_\gamma)$ exists and is independent of the precise choice of the net $(x_\gamma)_{\gamma \in \Gamma}$, then $\overline{f}(x) := \lim_{\gamma \in \Gamma} f(x_\gamma)$ defines a continuous function on $X \setminus A$. 
\end{lem}
\begin{proof}
    Both the statement and the proof are a straightforward modification of Bourbaki's extension theorem, \cite[Page 81, Theorem I]{Bourbaki1987}.
\end{proof}
\begin{lem}\label{lemma:top2}
    Let $X, A, Y, f$ and $\overline{f}$ as in the previous lemma. Further, let $X'$ be another topological space and $\iota: X \to X'$ continuous and surjective such that $\iota|_A$ is injective and bicontinuous with range dense in $X'$. If for every $x \in X' \setminus \iota(A)$ and every $x_0, y_0 \in \iota^{-1}(x)$ we have $\overline{f}(x_0) = \overline{f}(y_0)$, then for each $x \in X' \setminus \iota(A)$ and each net $(x_\gamma)_{\gamma \in \Gamma} \subset \iota(A)$ with $x_\gamma \to x$ the limit $\lim_{\gamma \in \Gamma} g(x_\gamma)$ depends only on $x$ and not on the precise choice of the net. Further, in this case $g = f \circ (\iota|_A)^{-1}: \iota(A) \to Y$ gives rise to a continuous function $\overline{g}$ from $X' \setminus \iota(A)$ to $Y$ in the sense of the previous lemma.
\end{lem}
\begin{proof}
    The statement is a consequence of the previous lemma, as for any net $(x_\gamma)_{\gamma \in \Gamma} \subset \iota(A)$ converging to $x \in X' \setminus \iota(A)$, the value of $\lim_{\gamma \in \Gamma} g(x_\gamma) = \lim_{\gamma \in \Gamma} f(\iota^{-1}(x_\gamma))$ does not depend on the choice of the net, by assumption.
\end{proof}

Recall that for $f \in \mathrm{BUC}(G)$ the limit function $\alpha_x(f)$ can be expressed as $\alpha_x(f)(y) = x(\beta_-(\alpha_y(f)))$. When $f \in \mathcal A$ and $\mathcal A$ is not $\beta_-$-invariant, then this expression makes sense only if $x \in \mathcal M(\beta_-(\mathcal A))$ instead of $x \in \mathcal M(\mathcal A)$. This is the reason for labeling limit functions in the following by points from $\mathcal M(\beta_-(\mathcal A))$.

For a given function $f \in L^\infty(G)$ and $\mathcal A$ as above, we say that the limit functions of $f$ exist on $\partial_{\beta_-(\mathcal A)}(G)$ with respect to the topology $\tau$ if for every $x \in \partial_{\beta_-(\mathcal A)}$ and any net $(x_\gamma) \subset G$ with $x_\gamma \to x$ the limit $\lim_{\gamma \in \Gamma} \alpha_{x_\gamma}(f)$ exists with respect to the topology $\tau$. If this is the case, the previous lemma shows that the limit functions $\alpha_x(f) := \lim_{\gamma \in \Gamma} \alpha_{x_\gamma}(f)$ depend continuously (with respect to the topology $\tau$) on $x \in \partial_{\beta_-(\mathcal A)}$. Note that the previous lemma implies also that the limit $\alpha_x(f)$, $x \in \mathcal M(\beta_-(\mathcal A))$, is independent of the choice of the particular net $(x_\gamma)_{\gamma \in \Gamma} \subset G$ converging to $x$.

We continue with some more definitions:\label{def:Aj}
\begin{align*}\nomenclature{$\mathcal A_j(\tau)$}{cf.\ p.\ \pageref{def:Aj}}
\mathcal A_3(\tau) &:= \{ f \in L^\infty(G): ~\text{the limit functions of $f$ exist on $\partial_{\beta_-(\mathcal A)}(G)$ with respect to $\tau$}\},\\
\mathcal A_2(\tau) &:= \{ f \in \mathcal A_3(\tau): ~\alpha_x(f) \in \operatorname{BUC}(G) \text{ for } x \in \partial_{\beta_-(\mathcal A)}(G)\},\\
\mathcal A_1(\tau) &:= \{ f \in \mathcal A_2(\tau): \{ \alpha_x(f): x \in \partial_{\beta_-(\mathcal A)}(G)\} \text{ is unif.\ equicont.}\},\\
\mathcal A_0(\tau) &:= \mathcal A.
\end{align*}
Here, we say that a family $F \subset \operatorname{BUC}(G)$ is uniformly equicontinuous if for every $\varepsilon > 0$ there exists a neighborhood $O$ of the unit element $e$ such that for every $y \in O$ and $f \in F$ it is
\begin{align*}
\| \alpha_y(f) - f\|_\infty < \varepsilon.
\end{align*}
Of course, the definition of $\mathcal A_0(\tau)$ is solely for cosmetic reasons for a more unified notation in the following results, as the space is entirely independent of the topology $\tau$. 

Note that the spaces should be understood as follows: $\mathcal A_3(\tau)$ is the space where the limit functions exist w.r.t.\ the topology $\tau$ and can be indexed by the boundary points coming from the algebra $\beta_-(\mathcal A)$. $\mathcal A_2(\tau)$ is the subspace where all the limit functions are nice (i.e., in $\operatorname{BUC}(G)$), and $\mathcal A_1(\tau)$ is the space where the limit functions form a ``nice family''.
We add the following lemma on the structure of these spaces, where $\operatorname{BUC}(G)_j(\tau)$ denotes the space $\mathcal A_j(\tau)$ with $\mathcal A = \operatorname{BUC}(G)$.
\begin{lem}\label{lem:mainthm1}
    Let $\tau$ and $\mathcal A$ as above and $j \in\{ 1, 2, 3\}$. Then, the following holds true:
    \begin{align*}
        \mathcal A_j(\tau) = \mathrm{BUC}(G)_j(\tau_{w^\ast}) \cap \mathcal A_3(\tau_{w^\ast}) \cap \mathrm{BUC}(G)_3(\tau).
    \end{align*}
\end{lem}
\begin{proof}
    We first consider the case $j = 3$. Then, each of the spaces occurring here is contained in $\mathrm{BUC}(G)_3(\tau_{w^\ast})$ so that we only need to verify that $\mathcal A_3(\tau) = \mathcal A_3(\tau_{w^\ast}) \cap \operatorname{BUC}(G)_3(\tau)$. Since $\tau$ is finer than $\tau_{w^\ast}$, it is clear that $\mathcal A_3(\tau) \subset \mathcal A_3(\tau_{w^\ast})$. Further, since $\mathcal A$ is a subalgebra of $\mathrm{BUC}(G)$, we also have $\mathcal A_3(\tau) \subset \mathrm{BUC}(G)_3(\tau)$. Hence, we have proven that $\mathcal A_3(\tau) \subseteq \mathcal A_3(\tau_{w^\ast}) \cap \mathrm{BUC}(G)_3(\tau)$. 
    
    We now come to the other inclusion. Since $\mathcal A$ is a unital $C^\ast$-subalgebra of $\operatorname{BUC}(G)$, we can consider a continuous surjection $\iota: \sigma G \to \mathcal M(\beta_-(\mathcal A))$ which fixes $G$ (which is, in a Gelfand-theoretic sense, the dual of the embedding map $\beta_-(\mathcal A) \hookrightarrow \operatorname{BUC}(G)$). Now note that if the limit functions of $f \in L^\infty(G)$ exist with respect to $\tau$ on $\partial G$ and with respect to $\tau_{w^\ast}$ on $\partial_{\beta_-(\mathcal A)}(G)$, for $x \in \partial_{\beta_-(\mathcal A)}(G)$ and $y_1, y_2 \in \iota^{-1}(x)$ the limit functions $\alpha_{y_1}(f)$ and $\alpha_{y_2}(f)$ have to agree (by definition of what it means that the limit functions exist on $\partial_{\beta_-(\mathcal A)}(G)$). Hence, an application of Lemma \ref{lemma:top2} shows that the limit functions of $f$ exist also with respect to $\tau$ on $\partial_{\beta_-(\mathcal A)}(G)$, which provides the other inclusion. This settles the case $j = 3$. 

    Now that the case $j = 3$ is proven, the case $j = 2$ is straightforward, as the property of all limit functions being uniformly continuous is entirely reflected by membership in $\mathrm{BUC}(G)_2(\tau_{w^\ast})$, i.e., a function $f \in \mathcal A_3(\tau)$ is contained in $\mathcal A_2(\tau)$ if and only if it is contained in $\mathrm{BUC}(G)_2(\tau_{w^\ast})$.

    The last case $j = 1$ is now dealt with analogously, as the property of the family of limit functions being uniformly equicontinuous is entirely captured by membership in $\mathrm{BUC}(G)_1(\tau_{w^\ast})$.
\end{proof}

\begin{thm}\label{mainthm:1}
Let $S \subset L^1(G)$ be a regular subset, $\tau$ and $\mathcal A$ as above, $j \in \{ 0, 1, 2, 3\}$ and $f \in \operatorname{BUC}(G)_j(\tau)$. Further, assume that $C_0(G) \subset \mathcal A$. Then, the following holds:
\begin{align*}
f \in \mathcal A_j(\tau) \Leftrightarrow g \ast f \in \mathcal A \text{ for every } g \in S.
\end{align*}
\end{thm}
\begin{proof}
Note that the case $j = 0$ was already covered by Theorem \ref{thm:bucinv} above - we include it to point out the analogy. Therefore, we may exclude the case $j = 0$ in the proof.

Note that, by assumption of the theorem, we have that $f \in \mathrm{BUC}(G)_j(\tau)$. By Lemma \ref{lem:mainthm1} this is equivalent to $f\in \mathrm{BUC}(G)_j(\tau_{w^\ast}) \cap \mathrm{BUC}(G)_3(\tau)$. Hence, by Lemma \ref{lem:mainthm1}, we see that:
\begin{align*}
    f \in \mathcal A_j(\tau) \Leftrightarrow f \in \mathcal A_3(\tau_{w^\ast}).
\end{align*}
 This reduces the proof of the theorem to the case where $j = 3$ and $\tau = \tau_{w^\ast}$. Hence, we only need to prove the following:
\begin{align*}
\text{the limit functions of $f$ exist on $\partial_{\beta_-(\mathcal A)}(G)$ with respect to $\tau_{w^\ast}$} \Leftrightarrow g \ast f \in \mathcal A \text{ for every } g \in S.
\end{align*}
Let $\iota: \sigma G \to \mathcal M(\beta_-(\mathcal A))$ again be the surjection already considered in the proof of Lemma \ref{lem:mainthm1}. As a consequence of Lemma \ref{lemma:top2}, the limit functions of $f$ exist on $\partial_{\beta_-(\mathcal A)}(G)$ with respect to $\tau_{w^\ast}$ if and only if they exist on $\partial G$ with respect to $\tau_{w^\ast}$ and for every $x \in \partial_{\beta_-(\mathcal A)}(G)$ and $y_1, y_2 \in \iota^{-1}(x)$ it is $\alpha_{y_1}(f) = \alpha_{y_2}(f)$. 

A similar argument works for checking membership of $f \in \operatorname{BUC}(G) = C(\sigma G)$ in $\mathcal A = C(\mathcal M(\mathcal A))$: $f \in \mathcal A$ if and only if $\alpha_{y_1}(f) = \alpha_{y_2}(f)$ for every $y_1, y_2 \in \iota^{-1}(x)$, where $x \in \partial_{\beta_-(\mathcal A)}(G)$. Indeed, if $f \in \mathcal A$, then the other statement is clear. If, on the other hand, $\alpha_{y_1}(f) = \alpha_{y_2}(f)$ for every $y_1, y_2 \in \varphi^{-1}(x)$, where $x \in \partial_{\beta_-(\mathcal A)}(G)$ is arbitrary, then $x \mapsto \alpha_x(f)$ is a continuous $\operatorname{BUC}(G)$-valued function on $\partial_{\beta_-(\mathcal A)}(G)$. Hence, the function $\tilde{g}(x) = \alpha_x(f)(e)$ is a continuous function on $\partial_{\beta_-(\mathcal A)}(G)$, hence $g(x) = \tilde{g}(\rho x)$ is in $C(\partial_{\mathcal A} G)$. Since we assumed $C_0(G) \subset \mathcal A$, $\partial_{\mathcal A} G$ is a closed subset of $\mathcal M(\mathcal A)$ such that we can extend $g$ to a continuous function $g_0$ on $\mathcal M(\mathcal A)$ (by Tietze's extension theorem), i.e., $g_0 \in \mathcal A$. Since $g_0$ has limit functions equal to those of $f$, we see that $g_0 - f \in C_0(G)$. Using again that $C_0(G) \subset \mathcal A$, we obtain that $f \in \mathcal A$.

These equivalences will be our main tools to check the implications. Hence, in the following we let $x \in \partial_{\beta_-(\mathcal A)}(G)$ and fix $y_1, y_2 \in \varphi^{-1}(x)$.

``$\Rightarrow$'': We need to show that $\alpha_{y_1}(g \ast f) = \alpha_{y_2}(g \ast f)$ for all $y_1, y_2 \in \iota^{-1}(x), ~x \in \partial_{\beta_-(\mathcal A)}(G)$. This is a simple consequence of the following:
\begin{align*}
\alpha_{y_1}(g \ast f) = g \ast \alpha_{y_1}(f) = g \ast \alpha_x(f) = g \ast \alpha_{y_2}(f) = \alpha_{y_2}(g \ast f).
\end{align*}

``$\Leftarrow$'': We need to show that $\alpha_{y_1}(f) = \alpha_{y_2}(f)$ for all $y_1, y_2 \in \iota^{-1}(x), ~x \in \partial_{\beta_-(\mathcal A)}(G)$. For $g \in S$, we know that $g \ast f \in \operatorname{BUC}(G)$, hence $g \ast f$ extends to a continuous function on $\sigma G$. Therefore, we can evaluate $g \ast f$ at $\rho y_1$ and get (for $(y_\gamma)$ a net in $G$ converging to $y_1$):
\begin{align*}
g \ast f(\rho y_1) &= \lim_\gamma g \ast f(\rho y_\gamma) =\lim_\gamma g \ast \alpha_{y_\gamma}(f)(e)\\
&= \lim_\gamma \int_G g(x) \alpha_x (\alpha_{y_\gamma}(f))(e)~dx\\
&= \lim_\gamma \int_G g(x) \alpha_{y_\gamma}(f)(\rho x)~dx\\
&= \lim_\gamma \int_G \beta_- g(x) \alpha_{y_\gamma}(f)(x)~dx\\
&= \lim_\gamma \langle \beta_- g, \alpha_{y_\gamma}(f)\rangle\\
&= \langle \beta_- g, \alpha_{y_1} (f)\rangle
\end{align*}
Since we even assumed $g \ast f \in \mathcal A$, we have $g \ast f(\rho y_1) = g\ast f(\rho y_2)$, hence:
\begin{align*}
\langle \beta_- g, \alpha_{y_1}(f)\rangle = g \ast f(\rho y_1) = g\ast f(\rho y_2) = \langle \beta_- g, \alpha_{y_2}(f)\rangle.
\end{align*}
From this, we immediately get
\begin{align*}
\langle h, \alpha_{y_1}(f)\rangle = \langle h, \alpha_{y_2}(f)\rangle
\end{align*}
for every $h$ from the space spanned by the translates of functions from $\beta_- S$. But $S$ is regular if and only if $\beta_- S$ is, so $h$ is taken from a dense subspace. By a standard density argument, we obtain that
\begin{align*}
\langle h, \alpha_{y_1}(f)\rangle = \langle h, \alpha_{y_2}(f)\rangle
\end{align*}
for every $h \in L^1(G)$. This shows that the limit functions of $f$ exist on $\partial_{\beta_-(\mathcal A)}(G)$ with respect to $\tau_{w^\ast}$ as desired.
\end{proof}
\begin{rem}\label{remark:C0}
    We want to add a brief word of explanation on the assumption that $C_0(G) \subset \mathcal A$. If one drops the assumption that $\mathcal A$ separates the points of $G$ (and hence $C_0(G)$ can no longer be contained in $\mathcal A$), then $H := \{ h \in G: ~\alpha_h(f) = f$ for every $f \in \mathcal A\}$ is a closed subgroup of $G$. Replacing the role of $\operatorname{BUC}(G)$ in the above result by
    \begin{align*}
        \operatorname{BUC}(G)_H := \{ f \in \operatorname{BUC}(G): ~\alpha_h(f) = f \text{ for all } h \in H\}
    \end{align*}
    and the role of $C_0(G)$ by
    \begin{align*}
        \{ f \in \operatorname{BUC}(G)_H: ~f(x) = 0 \text{ for every } x \in \mathcal M(\operatorname{BUC}(G)_H) \setminus H\},
    \end{align*}
    one can prove a more general result than the above without substantial difference in the method of proof. Further, all the results we discuss below can be extended to this more general situation. Since the results and proofs we present here are, even without discussing the role of the subgroup $H$ in detail, already filled with enough technical details, we omit this generalization throughout the paper. Nevertheless, the reader should keep in mind that this obstacle can also be overcome without introducing significant new ideas.
\end{rem}
For the second main result, we will also need a translation-invariant closed ideal $\mathcal I$ of $\mathcal A$. Note that the action of $G$ on itself, which we denote by $\varsigma_y(x)$, \nomenclature{$\varsigma_y(x)$}{Action of $G$ on itself, resp.\ its extension to $\mathcal M(\mathcal A)$, p.\ \pageref{def:varsigma}}\label{def:varsigma}$\varsigma_y(x) = x-y$ for $x, y \in G$, can be extended to $\mathcal M(\mathcal A)$: $\varsigma_y(x)(f) = x(\alpha_y(f))$ for $x \in \mathcal M(\mathcal A)$, $y \in G$ and $f \in \mathcal A$. As is well-known, closed ideals of $\mathcal A \cong C(\mathcal M(\mathcal A))$ correspond to closed subsets $\tilde{I} \subset \mathcal M(\mathcal A))$ via
\begin{align*}
\mathcal I = \{ f \in C(\mathcal M(\mathcal A)): ~f(x) = 0 \text{ for every } x \in \tilde{I}\}.
\end{align*}
Since $\mathcal I$ is assumed $\alpha$-invariant, we obtain for every $y \in G$:
\begin{align*}
\alpha_y(f) \in \mathcal I \Rightarrow \alpha_y(f)(x) = f(\varsigma_y(x)) = 0 \text{ for every } x \in \tilde{I}.
\end{align*}
Hence, $\tilde{I}$ is necessarily $\varsigma$-invariant (indeed, every closed, $\varsigma$-invariant subset induces a closed and translation-invariant ideal in this way). In particular, we get (since $G$ is dense in $\mathcal M(\mathcal A)$):
\begin{align*}
\mathcal I &= \{ f \in \mathcal A: \forall x \in \tilde{I}, y \in G: \alpha_y(f)(x) = \alpha_{\rho x}f(-y) = 0\}\\
&= \{ f \in \mathcal A: \alpha_{\rho x}(f) = 0 \text{ for every } x \in \tilde{I}\}\\
&= \{ f \in \mathcal A: \alpha_x(f) = 0 \text{ for every } x \in I\},
\end{align*}
where we set $I := \rho \tilde{I}$, which is clearly again a $\varsigma$-invariant and closed subset of $\mathcal M(\beta_-(\mathcal A))$.

We now set for $j \in \{0, 1, 2, 3\}$:\label{def:Ij}
\begin{align*}
\nomenclature{$\mathcal I_j(\tau)$}{cf.\ p.\ \pageref{def:Ij}}\mathcal I_j(\tau) := \{ f \in \mathcal A_j(\tau): \alpha_x(f) = 0 \text{ whenever } x \in I\}.
\end{align*}
Our second main result is now the following:
\begin{thm}\label{mainthm:2}
Let $\mathcal A, \mathcal I$ and $\tau$ as above, $j \in \{0, 1, 2, 3\}$. Further, let $S \subset L^1(G)$ be a regular subset and $a \in \mathcal A_j(\tau)$. Then, for $f \in \mathcal  A_j(\tau)$ the following are equivalent:
\begin{enumerate}[(i)]
\item $f - a \in \mathcal I_j(\tau)$;
\item $\alpha_x(f) = \alpha_x(a)$ for every $x \in I$;
\item $g \ast f - g \ast a \in \mathcal I$ for every $g \in S$;
\item $g \ast f - g \ast a \in \mathcal I$ for every $g \in L^1(G)$.
\end{enumerate}
\end{thm}
\begin{proof}
By Wiener's approximation theorem, the equivalence of $(iii)$ and $(iv)$ is straightforward to verify. Also, the equivalence of $(i)$ and $(ii)$ is clear by the definition of $\mathcal I_j(\tau)$. By linearity we may now assume $a = 0$.

$(i) \Rightarrow (iii)$: For $x \in I$ we clearly have
\begin{align*}
\alpha_x(g \ast f) = g \ast \alpha_x(f) = g\ast 0 = 0.
\end{align*}
$(iii) \Rightarrow (i)$: For each $x \in I$ and $g \in S$ it is
\begin{align*}
0 = \alpha_x(g \ast f)(e) = \langle \beta_- g, \alpha_x(f)\rangle.
\end{align*}
By regularity of $S$ and $\varsigma$-invariance of $I$, this can easily be extended to
\begin{align*}
\langle h, \alpha_x(f)\rangle = 0, \quad \text{any } h \in L^1(G).
\end{align*}
But this clearly means $\alpha_x(f) = 0$ for any $x \in I$, hence $f \in \mathcal I_j(\tau)$.
\end{proof}
Note that for the choice $\mathcal A = \operatorname{BUC}(G)$, $a = const$, $j = 3$ and $\mathcal I = C_0(G)$ we are in the situation of Wiener's Tauberian theorem.

\subsection{Some consequences}
Now we state some corollaries of our results in the preceding section. Letting in Theorem \ref{mainthm:1} $\tau = \tau_{w^\ast}$, we obtain $\operatorname{BUC}(G)_3(\tau_{w^\ast}) = L^\infty(G)$ and hence:
\begin{cor}
Let $\mathcal A$ be a unital $C^\ast$-subalgebra of $\operatorname{BUC}(G)$ containing $C_0(G)$. Then,
\begin{align*}
\{ f \in L^\infty(G): ~g \ast f \in \mathcal A \text{ for every } g \in L^1(G)\} = \mathcal A_3(\tau_{w^\ast}).
\end{align*}
\end{cor}
In particular, the condition that $g \ast f \in \mathcal A$ for every $g \in L^1(G)$ is a topological condition at infinity. This corollary is, in a sense, a topological version of Wiener's Tauberian theorem.

As the next application, we want to obtain Pitt's extension as a Corollary to Theorem \ref{mainthm:2}. The natural definition of $\operatorname{SO}(G)$ is given in terms of a metric. Since this space is supposed to consist of functions oscillating slowly at infinity, describing such a property needs a proper metric. $G$ can be endowed with a proper translation-invariant compatible metric if it is second countable, which in turn is equivalent to $G$ being first countable and $\sigma$-compact (this is Struble's theorem, cf.\ Theorem 2.B.4 in \cite{Cornulier_Harpe2016}).
In the presence of such a metric $d$, we let:
\begin{align*}
\operatorname{SO}(G) := \{ f &\in L^\infty(G): ~~\forall \varepsilon > 0~~ \exists\, R > 0,~~ \delta > 0: \\
&d(x, e) > R~~ \text{and}~~ d(y,e) < \delta~ \Rightarrow |f(x) - f(x-y)| < \varepsilon\}.
\end{align*}
More precisely, the condition above should hold in $x$ up to a set of measure zero (but for every such $y$ with no exceptional set). This motivates the following general definition for an arbitrary lca group $G$:\label{def:SO}
\begin{align*}
\nomenclature{$\operatorname{SO}(G)$}{Slowly oscillating functions on $G$, p.\ \pageref{def:SO}}\operatorname{SO}(G) := \{ f \in L^\infty(G): &\forall \varepsilon > 0 \ \exists K \subset G \text{ compact, } O \subset G \text{ nbhd.\ of } e:\\
&\forall y \in O: \|(f - \alpha_y(f))\chi_{K^c}\|_\infty < \varepsilon\}.
\end{align*}
Here, we wrote $K^c := G \setminus K$ for the complement of $K$ in $G$.

In principle, this definition asks for uniform continuity at infinity. We will see that this is exactly what $\operatorname{SO}(G)$ is about. We say that $f \in L^\infty(G)$ \emph{vanishes at infinity} (\nomenclature{$B_0(G)$}{Bounded functions on $G$ vanishing at infinity, p.\ \pageref{def:B0}}\label{def:B0}write $f \in B_0(G)$) if for every $\varepsilon > 0$ there is some compact $K \subset G$ such that $\| f\chi_{K^c}\|_\infty < \varepsilon$. Indeed, $\operatorname{SO}(G)$ can be described as follows:
\begin{prop}\label{decomp:SO}
Let $G$ be a lca group. Then, the following holds true:
\begin{align*}
    \mathrm{SO}(G) = \mathrm{BUC}(G) + B_0(G).
\end{align*}
\end{prop}
Before proving the above proposition, we need some auxiliary facts.

\begin{lem}\label{lem:SOCb}
Let $G$ be a lca group. Then, $C_b(G) \cap \operatorname{SO}(G) = \operatorname{BUC}(G)$.
\end{lem}
\begin{proof}
The inclusion ``$\supseteq$'' is clear. For the other inclusion, given $\varepsilon > 0$, use the $\operatorname{SO}(G)$ property outside $K$ and the fact that $f$ is continuous on $K$, which is compact, hence uniformly continuous on $K$.
\end{proof}
Before attempting the next characterization of $\operatorname{SO}(G)$, we need the following:
\begin{lem}\label{lemma:reconstr_boundary1}
Let $f \in L^\infty(G)$. Assume that:
\begin{itemize}
\item $\alpha_x(f) \in \operatorname{BUC}(G)$ for every $x \in \partial G$;
\item $\partial G \ni x \mapsto \alpha_x(f)(e)$ is continuous;
\item $\{ \alpha_x(f): ~x \in \partial G\}$ is uniformly equicontinuous.
\end{itemize}
Then, there exists some $g \in \operatorname{BUC}(G)$ such that $\alpha_x(f) = \alpha_x(g)$ for all $x \in \partial G$.
\end{lem}
\begin{proof}
By assumption, the function
\begin{align*}
\tilde{g}: \partial G \to \mathbb C, ~\tilde{g}(x) = \alpha_x(f)(e)
\end{align*}
is continuous. Since $\sigma G$ is normal and $\partial G$ is a closed subset, Tietze's extension theorem yields that there is some $g_0 \in C(\sigma G) \cong \operatorname{BUC}(G)$ with $g_0(x) = \tilde{g}(x)$ for each $x$ from the boundary. Set $g(x) = g_0(\rho x)$. Now, for $x \in \partial G$ and $y \in G$ we have that
\begin{align*}
    \alpha_x(f)(y) = \alpha_x(\alpha_{-y}(f))(e) = \alpha_{\varsigma_{-y}(x)}(f)(e) = g_0(\varsigma_{-y}(x)) = \alpha_{\varsigma_{-y}(x)}(\beta_-(g_0))(e) = \alpha_x(g)(y),
\end{align*}
finishing the proof.
\end{proof}
\begin{lem}\label{lemma:b0}
Let $f \in L^\infty(G)$ such that $\alpha_{x_\gamma}(f)\to 0$ uniformly on compact subsets whenever $G \ni x_\gamma \to x \in \partial G$. Then, $f \in B_0(G)$. In particular, with the notation of Section \ref{subsec:22} we have the following equalities of sets:
\begin{align*}
    B_0(G) = C_0(G)_1(\tau_{u.c.}) = C_0(G)_2(\tau_{u.c.}) = C_0(G)_3(\tau_{u.c.}).
\end{align*}
\end{lem}
\begin{proof}
We may without loss of generality assume that $G$ is not compact, otherwise there is nothing to prove.

Assume that $f \not \in B_0(G)$, i.e.\ there exists $\varepsilon > 0$ such that $\| f\chi_{K^c}\|_\infty > \varepsilon$ for every $K \subset G$ compact. Fix some compact neighborhood $\tilde{K}$ of $e \in G$. 

Denote by $\mathcal K$ the net of all compact neighborhoods of $e$, ordered increasingly by inclusion. For each $K \in \mathcal K$ we have $\| f\chi_{K^c}\|_\infty > \varepsilon$. Hence, there exists for each $K \in \mathcal K$ some $x_K \in K^c$ such that 
\begin{align*}
    \| f\chi_{K^c \cap (\tilde{K} + x_K)}\|_\infty \geq \varepsilon/2.
\end{align*}
Since the net $(x_K)_{K \in \mathcal K}$ eventually leaves every compact subset of $G$, there exists a subnet $(y_\gamma)_{\gamma \in \Gamma}$ (i.e, a net $(y_\gamma)_{\gamma\in \Gamma}$ together with a monotone final function $h: \Gamma \to \mathcal K$ such that $y_\gamma = x_{h(\gamma)}$ for all $\gamma \in \Gamma$) such that $y_\gamma \to y$ for some $y \in \partial G$. We now have for each $\gamma \in \Gamma$ that
\begin{align*}
    \| f \chi_{y_\gamma + \tilde{K}}\|_\infty \geq \| f\chi_{h(\gamma)^c \cap (\tilde{K} + y_\gamma)}\|_\infty \geq \varepsilon/2.
\end{align*}
Therefore, $\alpha_{y_\gamma}(f)$ does not converge to $0$ uniformly on compact subsets.

Now, regarding the equalities of sets, it follows immediately from the definition that
\begin{align*}
    B_0(G) \subseteq C_0(G)_1(\tau_{u.c.}) = C_0(G)_2(\tau_{u.c.}) = C_0(G)_3(\tau_{u.c.}).
\end{align*}
The inclusion $B_0(G) \supseteq C_0(G)_3(\tau_{u.c.})$ follows now from the first part of the proof.
\end{proof}
\begin{prop}
Let $f \in L^\infty(G)$. Then, $f \in \operatorname{BUC}(G) + B_0(G)$ if and only if $f$ satisfies the following properties:
\begin{itemize}
\item $\alpha_{x_\gamma}(f) \to \alpha_x(f)$ uniformly on compact sets whenever $x \in \partial G$ and $(x_\gamma) \subset G$ converging to $x$;
\item $\alpha_x(f) \in \operatorname{BUC}(G)$ whenever $x \in \partial G$;
\item The family $\{ \alpha_x(f): x \in \partial G\}$ is uniformly equicontinuous.
\end{itemize}
In particular,
\begin{align*}
    \operatorname{BUC}(G)_1(\tau_{u.c.}) = \mathrm{BUC}(G) + B_0(G).
\end{align*}
\end{prop}
\begin{proof}
Let $f = f_1 + f_2 \in \operatorname{BUC}(G) + B_0(G)$. Since $\alpha_{x_\gamma}(f) = \alpha_{x_\gamma}(f_1) + \alpha_{x_\gamma}(f_2)$ and $B_0(G) = C_0(G)_3(\tau_{u.c.})$ (cf.~Lemma \ref{lemma:b0}) we see that $\alpha_{x_\gamma}(f) \to \alpha_x(f) = \alpha_x(f_1)$ uniformly on compact subsets whenever $G \ni x_\gamma \to x \in \partial G$. In particular, $\alpha_x(f) \in \operatorname{BUC}(G)$. Further, for any $y \in G$ it is:
\begin{align*}
|\alpha_y(\alpha_x(f))(z) - \alpha_x(f)(z)| &= |\alpha_x(\alpha_y(f))(z) - \alpha_x(f)(z)|\\
&=  |x(\beta_- \alpha_z(\alpha_y(f) - f))|\\
&\leq \| x\| \| \alpha_z(\alpha_y(f) - f)\|\\
&= \| x\| \| \alpha_y(f) - f\|\\
&\leq \varepsilon
\end{align*}
for $y$ in an appropriate neighborhood of $e$. Since this estimate is uniform in $z$, we get that the family of limit functions is uniformly equicontinuous. Similarly, one shows that the family is uniformly bounded. This shows one direction of the statement.

Let $f \in L^\infty(G)$ satisfy the assumptions in the theorem. Once we show that the map $\partial G \ni x \mapsto \alpha_x(f)(e)$ is continuous, Lemma \ref{lemma:reconstr_boundary1} gives us a function $g \in \operatorname{BUC}(G)$ such that $\alpha_x(f-g) = 0$ for $x \in \partial G$. The previous lemma then shows $h := f-g \in B_0(G)$. Thus, $f = g+h \in \operatorname{BUC}(G) + B_0(G)$.

Therefore, it remains to show the continuity of $\partial G \ni x \mapsto \alpha_x(f)(e)$. Fix $x \in \partial G$ and let $(x_\gamma) \subset \partial G$ be a net converging to $x$. Fix $\varepsilon > 0$. Since the family of limit functions $\alpha_y(f)$, $y \in \partial G$ is uniformly bounded and uniformly equicontinuous, there exists some $g \in L^1(G)$ (e.g., taken from some appropriate approximate unit $(\psi_\theta) \subset L^1(G)$) such that for all such $y$:
\begin{align*}
|\langle \beta_- g, \alpha_y(f)\rangle - \alpha_y(f)(e)| =  |g \ast \alpha_y(f)(e) - \alpha_y(f)(e)| < \varepsilon/3. 
\end{align*}
Further, since $y \mapsto \alpha_y(f)$ is continuous in weak$^\ast$ topology, there is a $\gamma_0$ such that for all  $\gamma \geq \gamma_0$:
\begin{align*}
|\langle \beta_- g, \alpha_{x_\gamma}(f) - \alpha_x(f)\rangle| < \varepsilon/3.
\end{align*}
Together, we obtain for all $\gamma \geq \gamma_0$:
\begin{align*}
|\alpha_{x_\gamma}(f)(e) - \alpha_{x}(f)(e)| < \varepsilon.
\end{align*}
Since $\varepsilon > 0$ was fixed but arbitrary, this shows that $\alpha_{x_\gamma}(f)(e) \to \alpha_x(f)(e)$, which finishes the proof.
\end{proof}
\begin{proof}[Proof of Proposition \ref{decomp:SO}]
By the previous proposition, the inclusion ``$\supseteq$'' is clear. Hence, assume $f \in \operatorname{SO}(G)$. and let $\psi_{\theta}, \theta \in \Theta$, be an approximate identity of $L^1(G)$ coming from a neighborhood base $O_\theta$, $\theta \in \Theta$ of $e$, i.e. $\psi_{\theta} = 1/|O_{\theta}| \chi_{O_{\theta}}$.

For proving that $f \in \mathrm{BUC}(G)_1(\tau_{u.c.})$, we can verify the properties from the previous proposition.
\begin{itemize}
\item First we show that for $x \in \partial G$ and $(x_\gamma)_\gamma \subset G$ with $x_\gamma \to x$ it holds true that $\alpha_{x_\gamma}(f) \to \alpha_x(f)$ pointwise almost everywhere: Fix $z \in G$ and $\varepsilon > 0$. Let $K$, $O$ as in the definition of $f \in \operatorname{SO}(G)$. Then, there is a $\gamma_0$ such that for all $\gamma > \gamma_0$ one has $z-x_\gamma \not \in K$. Then, for $\theta_0$ such that for $\theta > \theta_0$ it is $O_\theta \subset O$, we have
\begin{align*}
|f(z-x_\gamma) - \psi_\theta &\ast \alpha_{x_\gamma}f (z)|\\
&\leq \int_G |f(z-x_\gamma) - f(z - x_\gamma - y)| \psi_{\theta}(y)~dy\\
&\leq \varepsilon.
\end{align*}
Since 
\begin{align*}
\psi_\theta  \ast \alpha_{x_\gamma}f(z) = \langle \alpha_{z}\beta_- \psi_\theta, \alpha_{x_\gamma}(f)\rangle \to \langle \alpha_z \beta_- \psi_\theta, \alpha_x(f)\rangle,
\end{align*}
there exists a $\gamma_0' > \gamma_0$ such that
\begin{align*}
|\psi_\theta \ast \alpha_{x_\gamma}f(z) - \langle \alpha_z \beta_- \psi_\theta, \alpha_x(f)\rangle| < \varepsilon.
\end{align*}
In particular, for $\gamma, \gamma' > \gamma_0'$, it is
\begin{align*}
|f(z - x_\gamma) - f(z-x_{\gamma'})| \leq 2\varepsilon + |\langle \alpha_z \beta_- \psi_\theta, \alpha_{x_\gamma}(f) - \alpha_{x_{\gamma'}}(f)\rangle|.
\end{align*}
Since $\alpha_{x_\gamma}(f)$ and $\alpha_{x_{\gamma'}}(f)$ converge to the same limit in the weak$^\ast$ topology, we obtain that $(f(z-x_\gamma))_\gamma$ is a Cauchy net. Hence, the limit $\alpha_x(f) = \lim_{\gamma\in \Gamma} \alpha_{x_\gamma}(f)$ does not only exist in weak$^\ast$ topology but also withrespect to pointwise almost everywhere convergence.
\item Let $\varepsilon > 0$ and $K \subset G$ compact. Let $K_\varepsilon$ and $O_\varepsilon$ as in the definition of $f \in \operatorname{SO}(G)$. Then, for $\gamma_0$ large enough, we have for $z \in G$ that $z - x_\gamma \in K_\varepsilon^c$ whenever $\gamma > \gamma_0$, hence
\begin{align*}
|\alpha_{x_\gamma}f(z) - \alpha_{x_\gamma}f(z-y)| \leq \varepsilon, \quad \gamma > \gamma_0.
\end{align*}
In particular, since the left-hand side converges to $|\alpha_{x}f(z) - \alpha_xf(z-y)|$, we get
\begin{align*}
|\alpha_xf(z) - \alpha_xf(z-y)| \leq \varepsilon
\end{align*}
for every $z \in G$ and $y \in O_\varepsilon$. We get that $\alpha_x(f)$ is uniformly continuous. Further, we have chosen $O_\varepsilon$ independently of the precise point $x \in \partial G$, therefore the family $\{ \alpha_x(f): x \in \partial G\}$ is uniformly equicontinuous. It clearly is also bounded by $\| f\|_\infty$, as
\begin{align*}
|\alpha_x(f)(z)| = \lim_\gamma |f(z-x_\gamma)| \leq \| f\|_\infty.
\end{align*}
\item Finally, we have to prove that $\alpha_{x_\gamma}(f) \to \alpha_x(f)$ uniformly on compact subsets. This follows from a standard $3\varepsilon$ argument: Fix $\varepsilon > 0$ and $K \subset G$ compact. Let $e \in O \subset G$ open such that it serves as a neighborhood of $e$ in the uniform continuity of $\alpha_x(f)$, but also as the open neighborhood of the definition of $f \in \operatorname{SO}(G)$. Further, let $K_\varepsilon$ the compact set in the definition of $f \in \operatorname{SO}(G)$. For $\gamma_0$ sufficiently large, we have $K - x_\gamma \cap K_\varepsilon = \emptyset$ for every $\gamma > \gamma_0$. For such $\gamma$ fixed, cover $K$ (compact) by finitely many sets $O + y_j$, $j = 1, \dots, m$. By possibly enlarging $\gamma_0$, we can enfore $|\alpha_{x_\gamma}f(y_j) - \alpha_x f(y_j)| < \varepsilon$ for all $j$. Now, for $z \in K$, there is some $j$ with $z \in O + y_j$. Hence, 
\begin{align*}
|\alpha_{x_\gamma}&f(z) - \alpha_x(f)(z)| \\
&\leq |\alpha_{x_\gamma}f(z) - \alpha_{x_\gamma}f(y_j)| + |\alpha_{x_\gamma}f(y_j) - \alpha_xf(y_j)| \\
&\quad + |\alpha_xf(y_j) - \alpha_xf(z)| \\
&\leq 3\varepsilon,
\end{align*}
which establishes uniform convergence on compact sets.\qedhere
\end{itemize}
\end{proof}

Clearly, we have $\operatorname{BUC}(G) \cap B_0(G) = C_0(G)$. Hence, using Proposition \ref{mainthm:2} with $\tau$ the compact-open topology, $\mathcal A = \operatorname{BUC}(G)$ and $\mathcal I = C_0(G)$, we obtain Pitt's extension for arbitrary lca groups:
\begin{prop}[Pitt's extension theorem]
Let $f \in \operatorname{SO}(G)$ and $S \subset L^1(G)$ a regular subset. Then, $f \in B_0(G)$ if and only if $g \ast f \in C_0(G)$ for every $g \in S$.
\end{prop}
To end this discussion, we want to note the following characterization of $B_0(G)$, showing that it is in some sense dual to $\operatorname{BUC}(G)$:
\begin{lem} We set
\begin{align*}
    \mathcal E_0(G) := \{ f \in L^\infty(G): ~\hat{G} \ni \xi \mapsto \gamma_\xi(f) \text{ is continuous w.r.t. } \| \cdot\|_\infty\}.
\end{align*}
Then, one has $B_0(G) = \mathcal E_0(G)$. Here, the action of the dual group $\hat{G}$ is given by modulation:
\begin{align*}
\gamma_\xi(f)(x) = \xi(x)f(x), \quad \xi \in \hat{G}.
\end{align*}
\end{lem}
\begin{proof}
Let $f \in B_0(G)$ and $\varepsilon > 0$. Let $K \subset G$ compact according to $f$ being in $B_0(G)$. Then, there is a neighborhood $\hat{O}$ of $\hat{e}$ in $\hat{G}$ such that $|\xi(x) - 1| < \varepsilon$ for every $x \in K$ and $\xi \in \hat{O}$. For such $\xi$ and $x \in K$, it is:
\begin{align*}
|\xi(x)f(x) - f(x)| \leq \| f\|_\infty |\xi(x) - 1| \leq \varepsilon \| f\|_\infty.
\end{align*}
Further, for $x \in K^c$ it is:
\begin{align*}
|\xi(x) f(x) - f(x)| \leq 2 \varepsilon.
\end{align*}
This shows the inclusion``$\subseteq$''. Now, assume that $f \in L^\infty(G)$ is such that $\widehat{G} \ni \xi \mapsto \gamma_\xi(f)$ is continuous with respect to the supremum norm. For $g \in L^1(\widehat{G})$, we set
\begin{align*}
    g \# f := \int_{\widehat{G}} g(\xi) \gamma_\xi(f)~d\xi.
\end{align*}
Since $\xi \mapsto \gamma_\xi(f)$ is continuous, this integral exists as an element of $\mathcal E_0(G)$. Further, we have (where the pointwise evaluation should of course be understood in the almost everywhere-sense):
\begin{align*}
    g \# f(y) = \int_{\widehat{G}} g(\xi) \xi(y) f(y)~d\xi = \mathcal F(g)(-y) f(y).
\end{align*}
Since $g \in L^1(\widehat{G})$, we know that $\mathcal F(g) \in C_0(G)$, and this clearly implies $g \# f \in B_0(G)$ (since $f$ is bounded). Now, let $(\widehat{O})_{\widehat{O} \in \mathcal O}$ be a neighborhood base of relatively compact neighborhoods of $e_{\widehat{G}}$ (where the index set is ordered by inclusion) and set $g_{\widehat{O}} = \frac{1}{|\widehat{O}|}\chi_{\widehat{O}}$, which is an approximate unit of $L^1(\widehat{G})$. Then, we have $g_{\widehat{O}} \# f \in B_0(G)$ for each $\widehat{O}$ and further:
\begin{align*}
    \| g_{\widehat{O}} \# f - f\|_\infty &= \| \int_{\widehat{G}} g_{\widehat{O}}(\xi) [\gamma_\xi(f) - f]~d\xi\|_\infty\\
    &\leq \int_{\widehat{G}} |g_{\widehat{O}}(\xi)|\| \gamma_\xi(f) - f\|_\infty~d\xi.
\end{align*}
It is now standard to conclude that this converges to $0$ as $\widehat{O} \in \mathcal O$. This shows $g_{\widehat{O}} \# f \to f$ in supremum norm. Since $B_0(G)$ is easily seen to be closed with respect to the supremum norm, this concludes $f \in B_0(G)$. 
\end{proof}

We have introduced the spaces $\mathcal A_j(\tau)$ and $\mathcal I_j(\tau)$. We want to give a quick discussion on how these spaces are connected to different choices of $j$, $\tau$ and $\mathcal A$.
\begin{exs}
\begin{enumerate}
\item We have seen that $\operatorname{BUC}(G)_3(\tau_{w^\ast}) = L^\infty(G)$. For $G = \mathbb R^2$, $A = \{ (x, y) \in \mathbb R^2: y \leq 0\}$ and $f = \chi_A$, let $n_\gamma$ be a convergent subnet of the sequence $((n, 0))_{n \in \mathbb N}$, converging to $x \in \partial_{\operatorname{BUC}(\mathbb R^2)}(\mathbb R^2)$. Then, $\alpha_{n_\gamma}(f) \to \alpha_x(f)$, but since $\alpha_{n_\gamma}(f) = f$ for every $\gamma$, it is $\alpha_x(f) = f$. Since $f \not \in \operatorname{BUC}(\mathbb R^2)$, this shows $\operatorname{BUC}(\mathbb R^2)_3(\tau_{w^\ast}) \supsetneq \operatorname{BUC}(\mathbb R^2)_2(\tau_{w^\ast})$.
\item Let $A \subset \mathbb R^2$ be the set $A = \{ (x, y) \in \mathbb R^2: ~0 \leq y \leq \frac{1}{|x|}\}$ and let $f = \chi_A$. Let $n_\gamma$ be a convergent subnet of the sequence $((0, -n))_{n \in \mathbb N}$ converging to $x \in \partial_{\operatorname{BUC}(\mathbb R^2)}(\mathbb R^2)$. Then, for every compact subset $K$ containing $(0, 0)$, one has  
\begin{align*}
\sup_{y \in K} |\alpha_{n_\gamma}f(y)|  \geq |\alpha_{n_\gamma}f(0, 0)| = |f(0,n_\gamma)| = 1.
\end{align*}
Moreover, it is not difficult to see that $\alpha_z(f) = 0$ for every $z \in \partial_{\operatorname{BUC}(\mathbb R^2)}(\mathbb R^2)$. This shows that $\operatorname{BUC}(\mathbb R^2)_j(\tau_{w^\ast}) \supsetneq \operatorname{BUC}(\mathbb R^2)_j(\tau_{u.c.})$ for $j = 1, 2, 3$.
\item It is an open problem if there are examples of $\mathcal A_2(\tau) \supsetneq \mathcal A_1(\tau)$.
\item We clearly always have $\mathcal A_3(\tau) \supseteq \mathcal A_2(\tau) \supseteq \mathcal A_1(\tau) \supseteq \mathcal A_0(\tau)$. Further, for $\tau_1 \subseteq \tau_2$, it is $\mathcal A_j(\tau_2) \supseteq \mathcal A_j(\tau_1)$. Analogous statements hold for ideals. Again, there are examples (essentially as in 1.\ and 2.\ above) where $\mathcal I_j(\tau_2) \supsetneq \mathcal I_j(\tau_1)$ for $\tau_1 \subsetneq \tau_2$.
\end{enumerate}
\end{exs}

\section{Wiener's Tauberian theorem in Quantum Harmonic Analysis}\label{sec:quantum}
In this section, we will always let $\Xi$ be a locally compact abelian group and $m: \Xi \times \Xi \to \mathbb T$ be a separately continuous Heisenberg multiplier on $\Xi$, i.e., it satisfies the cocycle relation
\begin{align*}
m(x+y, z)m(x,y) = m(x, y+z)m(y,z)
\end{align*}
and the map $\Xi \ni x \mapsto m(x,(\cdot))$ is a topological isomorphism from $\Xi$ to $\widehat{\Xi}$, cf.\ \cite{Fulsche_Galke2023} for details. Under these assumptions, there exists (up to unitary equivalence) a unique irreducible projective representation $(U_x)_{x \in \Xi}$ of $\Xi$ on a Hilbert space $\mathcal H$  with $m$ as the cocycle. We further assume that $m$ satisfies $m(x,y) = m(-x,-y)$ for all $x, y \in \Xi$ such that there exists a self-adjoint unitary operator $R$ on $\mathcal H$, which is unique up to a factor $\pm 1$, such that $RU_x = U_{-x}R$. We will use the conventions of convolutions and Fourier transforms from quantum harmonic analysis, as described in \cite{Fulsche_Galke2023}. We recall some of these notions for convenience.

For any function $f: \Xi \to \mathbb C$ we define for $x \in \Xi$ (compatible with the conventions from the previous section):
\begin{align*}
\alpha_x f = f(\cdot - x), \quad \beta_-(f) = f(-\cdot).
\end{align*}
Analogously, we set for $A \in \mathcal L(\mathcal H)$: \label{def:shiftofoperator}
\begin{align*}
\nomenclature{$\alpha_x(A)$}{Shift of operator, p.\ \pageref{def:shiftofoperator}, or limit operator, p.\ \pageref{def:limitops}}\alpha_x(A) = U_x A U_x^\ast, \quad \nomenclature{$\beta_-(A)$}{cf. p.\ \pageref{def:shiftofoperator}}\beta_-(A) = RAR.
\end{align*}
Now, for $f, g \in L^1(\Xi)$ and $A, B \in \mathcal T^1(\mathcal H)$, the trace class operators on $\mathcal H$, we define the following convolutions:
\begin{align*}
f \ast g(x) &:= \int_\Xi f(y) g(x-y)~dy, \\
f \ast A &:= A \ast f := \int_\Xi f(y) \alpha_y(A)~dy, \\
A \ast B(x) &:= \tr(A \alpha_x(\beta_-(B))).
\end{align*}
We have $f \ast g \in L^1(\Xi)$, with $f \ast A \in \mathcal T^1(\mathcal H)$ and $A \ast B \in L^1(\Xi)$. Endowed with these operations as the product, $L^1(\Xi) \oplus \mathcal T^1(\mathcal H)$ turns into a commutative Banach algebra, cf.\ \cite[Section 4]{Fulsche_Galke2023} for details and further properties of these convolutions. 

The above convolutions may be extended, by the same formulas, to the case where one of the factors is no longer integrable/trace class, but only bounded. That is: When $f \in L^1(\Xi), ~g \in L^\infty(\Xi), ~A \in \mathcal T^1(\mathcal H)$ and $B \in \mathcal L(\mathcal H)$, then the convolutions are defined by the same formulas and satisfy
\begin{align*}
f \ast g \in L^\infty(\Xi), &\quad \| f \ast g\|_{L^\infty} \leq \| f \|_{L^1} \| g\|_{L^\infty}\\
f \ast B \in \mathcal L(\mathcal H), &\quad \| f \ast B\|_{op} \leq \| f\|_{L^1} \| B\|_{op}\\
A \ast g \in \mathcal L(\mathcal H), &\quad \| A \ast g\|_{op} \leq \| A\|_{\mathcal T^1} \| g\|_\infty\\
A \ast B \in L^\infty(\Xi), &\quad \| A \ast B\|_{\infty} \leq \| A\|_{\mathcal T^1} \| B\|_{op}.
\end{align*}
We want to emphasize that the third of these estimates hinges on the correct normalization of the Haar measure; a ``wrong'' normalization leads to a constant factor on the right-hand side of the estimate, cf.\ again \cite{Fulsche_Galke2023} for details.

Another object that we will make use of in the following, is the Fourier transform of an operator. For $A \in \mathcal T^1(\mathcal H)$ we will write $\mathcal F_U(A)(\xi) = \tr(AU_\xi)$, where $\xi \in \Xi$. This Fourier transform is frequently referred to as the \emph{Fourier-Weyl} or \emph{Fourier-Wigner transform} of the operator. To some people, it might also be better known as the inverse of the group Fourier transform with respect to the projective representation $(U_x)_{x\in \Xi}$. For properties of this Fourier transform, we refer to \cite{werner84, Fulsche_Galke2023}.

One of the key results regarding these operator convolutions is the operator version of Wiener's approximation theorem, cf.\ \cite[Proposition 3.5]{werner84} for the initial version on a symplectic vector space and \cite[Theorem 6.30]{Fulsche_Galke2023} on locally compact abelian phase spaces. We repeat the statement:
\begin{thm}[Wiener's approximation theorem for operators]
Let $\mathcal R \subset \mathcal T^1(\mathcal H)$. Then, the following statements are equivalent. 
\begin{enumerate}
\item $\operatorname{span}\{ \alpha_x(A): ~A \in \mathcal R, ~x \in \Xi \}$ is dense in $\mathcal T^1(\mathcal H)$.
\item $L^1(\Xi) \ast \mathcal R$ is dense in $\mathcal T^1(\mathcal H)$.
\item $\mathcal T^1(\mathcal H) \ast \mathcal R$ is dense in $L^1(\Xi)$.
\item $\{ A \ast A: ~A \in \mathcal R\}$ is a regular family in $L^1(\Xi)$ in the sense of Wiener's classical approximation theorem.
\item $\cap_{A \in \mathcal R} \{ \xi \in \Xi: ~\mathcal F_U(A)(\xi) = 0\} = \emptyset$.
\item For $B \in \mathcal L(\mathcal H)$, $A \ast B = 0$ for every $A \in \mathcal R$ implies $B = 0$.
\item For $f \in L^\infty(\Xi)$, $A \ast f = 0$ for every $A \in \mathcal R$ implies $f = 0$.
\end{enumerate}
\end{thm}
A family $\mathcal R \subset \mathcal T^1(\mathcal H)$ satisfying the properties of the above theorem will be called a \emph{regular family}.
 
Note, that $x \mapsto \alpha_x(A)$ acts strongly continuous on the trace class $\mathcal T^1(\mathcal H)$. Observe that this continuity does not hold for the operator norm for every $A \in \mathcal L(\mathcal H)$. The class of such operators is the operator analogue of $\operatorname{BUC}(\Xi)$, which we denote by:\label{def:C1H}
\begin{align*}\nomenclature{$\mathcal C_1(\mathcal H)$}{Uniformly continuous operators, p.\ \pageref{def:C1H}}
\mathcal C_1(\mathcal H) = \{ A \in \mathcal L(\mathcal H): ~x \mapsto \alpha_x(A) \text{ is } \| \cdot\|_{op}\text{-cont.}\}.
\end{align*}
In \cite{werner84}, the concept of corresponding spaces was introduced, which is straightforward to adapt to the case of locally compact abelian phase space. Given two $\alpha$-invariant subspaces $\mathcal D_0\subset L^\infty(\Xi)$ and $\mathcal D_1\subset \mathcal L(\mathcal H)$, we say that these are \emph{corresponding spaces} if $\mathcal T^1(\mathcal H) \ast \mathcal D_0 \subset \mathcal D_1$ and $\mathcal T^1(\mathcal H) \ast \mathcal D_1 \subset \mathcal D_0$. Here is the main result regarding corresponding spaces:
\begin{thm}[Correspondence theorem]\label{theorem:correspondence}
Let $\mathcal D_0 \subset L^\infty(\Xi)$ and $\mathcal D_1 \subset \mathcal L(\mathcal H)$ be $\alpha$-invariant. Further, let $\mathcal R \subset \mathcal T^1(\mathcal H)$ be any regular family.
\begin{enumerate}
\item If $\mathcal D_0, \mathcal D_1$ are corresponding spaces, then so are $\overline{\mathcal D_0}, \overline{\mathcal D_1}$, their uniform closures.
\item Let $\mathcal D_0, \mathcal D_1$ be corresponding spaces. Then, $\mathcal R \ast \mathcal D_0$ is uniformly dense in $\mathcal D_1 \cap \mathcal C_1(\mathcal H)$ and $\mathcal R \ast \mathcal D_1$ is uniformly dense in $\mathcal D_0 \cap \operatorname{BUC}(\Xi)$.
\item Let $\mathcal D_0, \mathcal D_1$ be corresponding spaces. If $f \in \BUC$ such that $A \ast f \in \mathcal D_1$ for every $A \in \mathcal R$, then $f \in \overline{\mathcal D_0}$. If $B \in \mathcal C_1(\mathcal H)$ such that $A \ast B \in \mathcal D_0$ for every $A \in \mathcal R$, then $B \in \overline{\mathcal D_1}$.
\item Let $\mathcal D_0 \subset \BUC$ be closed in uniform topology. Then, there exists a unique $\alpha$-invariant and closed subspace of $\mathcal C_1(\mathcal H)$ corresponding to $\mathcal D_0$. This space is given by $\overline{\mathcal T^1(\mathcal H) \ast \mathcal D_0}$.
\item Let $\mathcal D_1 \subset \mathcal C_1(\mathcal H)$ be closed in operator norm topology. Then, there exists a unique $\alpha$-invariant and closed subspace of $\BUC$ corresponding to $\mathcal D_1$. This space is given by $\overline{\mathcal T^1(\mathcal H) \ast \mathcal D_1}$.
\end{enumerate}
\end{thm}
Indeed, having Wiener's approximation theorem available for locally compact abelian phase spaces, the proof of the correspondence theorem is a straightforward adaptation of the proof from \cite[Proposition 3.5]{werner84}, and hence we omit it. We also want to emphasize that the $L^p-\mathcal T^p$ version of the correspondence theorem, as well as the weak$^\ast$ version of the correspondence theorem, together with its applications on spectral synthesis (cf.\ \cite[Corollary 4.4]{werner84} or \cite{Fulsche_Rodriguez2023}) carries over to locally compact abelian phase spaces in the same way without any problems. We leave the details of this to the interested reader, as we will only make use of the result we explicitly spelled out above.

The two most important examples of these results, as well as their applications, are the following. The proof is, again, the same as for $\Xi = \mathbb R^{2n}$:
\begin{thm}\label{thm:standardcorrespondences}
 Let $\mathcal R \subset \mathcal T^1(\mathcal H)$ be a regular family.
\begin{enumerate}
\item $\BUC$ and $\mathcal C_1(\mathcal H)$ are corresponding spaces. In particular, $\mathcal C_1(\mathcal H) = \overline{\mathcal R \ast \BUC}$.
\item $C_0(\Xi)$ and $\mathcal K(\mathcal H)$ are corresponding spaces. In particular, an operator $B \in \mathcal L(\mathcal H)$ is compact if and only if $B \in \mathcal C_1(\mathcal H)$ and $A \ast B \in C_0(\Xi)$ for every $A \in \mathcal R$. 
\end{enumerate}
\end{thm}
We also want to mention the following result, the proof of which is entirely analogous to the case where $\Xi = \mathbb R^{2n}$, cf.\ \cite{Fulsche2020}.
\begin{thm}\label{thm:algebra}
Let $\mathcal D_0 \subset \BUC$ and $\mathcal D_1 \subset \mathcal C_1(\mathcal H)$ be closed, $\alpha$-invariant subspaces which correspond to each other in the sense of Theorem \ref{theorem:correspondence}. 
\begin{enumerate}
\item If $\mathcal D_0$ is a $C^\ast$-algebra, then so is $\mathcal D_1$.
\item If $\mathcal D_0$ is a $C^\ast$-algebra and $\mathcal I_0 \subset \mathcal D_0$ a closed, $\alpha$-invariant ideal, then the space $\mathcal I_1$ corresponding to $\mathcal I_0$ is a closed, two-sided ideal in $\mathcal D_1$.
\end{enumerate}
\end{thm}
Note that in \cite{Fulsche2020} it was additionally assumed that $\mathcal D_0$ is $\beta_-$-invariant. It is possible to remove this assumption by properly labeling the limit operators by boundary points from $\partial_{\beta_-(\mathcal D_0)}(\Xi)$ instead of $\partial_{\mathcal D_0}(\Xi)$, which was omitted in \cite{Fulsche2020}. The proof of the above result is now a straightforward adaptation of the methods in \cite{Fulsche2020}, cf.\ Proposition 3.17, Proposition 3.19 and Corollary 3.20 in that paper. Thus we shall not repeat this argument here.

\subsection{Limit operators and associated compatible families}

For $A \in \mathcal L(\mathcal H)$ and $B \in \mathcal T^1(\mathcal H)$, the map
\begin{align*}
\Xi \ni x \mapsto \langle \alpha_x(A), B\rangle = \tr( U_x A U_x^\ast B)
\end{align*}
is uniformly continuous. In particular, \label{def:limitops}$x \mapsto \alpha_x(A)$ extends to a map from $\sigma \Xi$ to $(\mathcal L(\mathcal H), \tau_{w^\ast}')$\label{def:wstartopoperators}, where $\tau_{w^\ast}'$ is the weak$^\ast$ topology arising from $\mathcal L(\mathcal H) \cong \mathcal T^1(\mathcal H)^\ast$. 

Indeed, for elements from $\mathcal C_1(\mathcal H)$ it is crucial that we obtain convergence towards the limit operators in a stronger topology, which will be SOT$^\ast$ convergence: convergence in the topology on $\mathcal L(\mathcal H)$ induced by the family of seminorms
\begin{align*}
\rho_{f, 1}(B) = \| Bf\|, \quad \rho_{f, 2}(B) = \| B^\ast f\|, \quad f \in \mathcal H.
\end{align*}
We will abbreviate this topology in the following by $\tau_{s^\ast}$\label{def:sotstar}\nomenclature{$\tau_{s^\ast}$}{Strong$^\ast$ operator topology, p.\ \pageref{def:sotstar}}. In contrast to the usual strong operator topology, the adjoint map is clearly continuous.
\begin{prop}
Assume that the phase space $(\Xi, m)$ is such that the projective representation $U_x$ is integrable. Let $B \in \mathcal C_1(\mathcal H)$ and $(x_\gamma) \subset \Xi$ a net converging to $x \in \partial \Xi$. Then, $\alpha_{x_\gamma}(B) \to \alpha_x(B)$ in SOT$^\ast$.
\end{prop}
In the following proof, we denote by $\mathcal H_1 \subset \mathcal H$ the space of integrable vectors, i.e., the space of all vectors $\varphi \in \mathcal H$ such that $x \mapsto \langle U_x\varphi, \varphi\rangle \in L^1(\Xi)$. Since we assume $\mathcal H_1$ to be non-trivial (this is the integrability condition of the representation), it is well-known that $\mathcal H_1$ is automatically dense in $\mathcal H$ and, indeed, also a subspace of $\mathcal H$. We want to note that in the case of the standard projective representation, i.e. $\Xi = G \times \widehat{G}$ for some lca group $G$ and $m((g, \xi), (h, \eta)) = \frac{a((g,\xi))a((h,\eta))}{a((g+h),(\xi \eta))} \overline{\langle g, \eta\rangle}$ with $a: G\times \widehat{G} \to \mathbb T$ continuous, $\mathcal H$ can be chosen as $L^2(G)$, $U_{(g, \xi)}f(t) = a((g,\xi)) \langle t, \xi\rangle f(t-g)$ and the space of integrable vectors $\mathcal H_1$ agrees with $S_0(G)$, the Feichtinger algebra (or, which is the same, the modulation space $M^1(G)$). See \cite{Feichtinger81, jakobsen18} for details on $S_0(G)$.

\begin{proof}
Clearly, it suffices to prove the statement for a subspace of $\mathcal C_1(\mathcal H)$ which is dense in uniform topology. Since $\mathcal H_1$ is dense in $\mathcal H$, and also by utilizing Theorem \ref{thm:standardcorrespondences}, it suffices to prove this convergence for $B = (\varphi \otimes \psi) \ast f$ when $\varphi, \psi \in \mathcal H_1$ and $f \in \BUC$. Note that we are using, by convention, the sesquilinear tensor product, which is linear in the first and antilinear in the second entry: $(\varphi \otimes \psi)(g) = \langle g, \psi\rangle \varphi$ for $g \in \mathcal H$. 

Since, for such $B$, the family of operators $\alpha_{x_\gamma}(B)$ is bounded in operator norm by $\| \alpha_{x_\gamma}(B)\| \leq \| \varphi\| \| \psi\| \|f\|_\infty$, it suffices to verify convergence of $\alpha_{x_\gamma}(B)g \to \alpha_x(B)g$ for $g$ from a dense subset of $\mathcal H$, e.g.\ for $g \in \mathcal H_1$. 

Let $\varepsilon > 0$. Since $x \mapsto \langle g, U_x \psi\rangle \in L^1(\Xi)$, there exists some compact subset $K \subset \Xi$ such that $\int_{K^c} |\langle g, U_x \psi\rangle|~dx < \varepsilon$. Then, we have:
\begin{align*}
\| (\varphi \otimes \psi) &\ast \alpha_{x_\gamma}(f)g  - (\varphi \otimes \psi)\ast \alpha_x(f) g\| \\
&\leq \int_\Xi |\alpha_{x_\gamma}(f)(y) - \alpha_x(f)(y)| |\langle g, U_y \psi\rangle| \| U_y \varphi\|~dy\\
&\leq 2\| f\|_\infty \| \varphi\| \varepsilon + \| \varphi\| \int_K |\alpha_{x_\gamma}(f)(y) - \alpha_x(f)(y)| |\langle g, U_y \psi\rangle|~dy
\intertext{Since $f \in \BUC$, we have that $\alpha_{x_\gamma}(f) \to \alpha_x(f)$ uniformly on compact subsets. Hence, there exists $\gamma_0$ such that for all $\gamma \geq \gamma_0$ we have $|\alpha_{x_\gamma}(f)(y) - \alpha_x(f)(y)| < \varepsilon$ for all $y \in K$. For such $\gamma$, it follows:}
&\leq \varepsilon ( 2\| f\|_\infty \| \varphi\| + \int_\Xi |\langle g, U_y\psi\rangle|~dy).
\end{align*}
Since $\varepsilon > 0$ was arbitrary, this shows that $\alpha_{x_\gamma}(B) \to \alpha_x(B)$ in SOT.

Since $\overline{f}$ is also in $\BUC$, the same argument works for the adjoint of the operator $B$, hence we see that $\alpha_{x_\gamma}(B) \to \alpha_x(B)$ in SOT$^\ast$.
\end{proof}

Based on the previous result, we make the following assumption for the remainder of this work:
\begin{assumption}
The projective representation $(U_x)_{x \in \Xi}$ of $\Xi$ on $\mathcal H$ is integrable.
\end{assumption}

We say that a set $S \subset \mathcal C_1(\mathcal H)$ is \emph{uniformly equicontinuous} if for every $\varepsilon > 0$ there exists a neighborhood $O$ of the unit $e \in \Xi$ such that for $x \in O$ it is
\begin{align*}
\| A  - \alpha_x(A)\|_{op} < \varepsilon \quad \text{ for every } A \in S.
\end{align*}
As in the case of functions, this notion is important for studying the limit operators of $A \in \mathcal L(\mathcal H)$, i.e., the operators $\alpha_x(A)$ for $x \in \partial \Xi$: For each $A \in \mathcal C_1(\mathcal H)$, the family $\{ \alpha_x(A): ~x \in \partial \Xi\}$ is uniformly equicontinuous. They form an important (and, as we will explain below, the only) example of the following class. Note that the following definition is taken from {\cite[Definition 6.5]{Fulsche2024}}, where it was set up for the special case $\Xi = \mathbb R^{2n}$.
\begin{defn}
A \emph{compatible family of limit operators} is a map $\omega: \partial \Xi \to \mathcal C_1(\mathcal H)$, satisfying the following properties:
\begin{enumerate}[(1)]
\item $\omega$ is continuous in the weak$^\ast$ topology.
\item For every $x \in \partial \Xi$ and $z \in \Xi$ one has:
\begin{align*}
\alpha_z(\omega(x)) = \omega(\varsigma_{-z}(x)).
\end{align*}
\item $\sup_{x \in \partial \Xi} \| \omega(x)\|_{op} < \infty$.
\item The family $\{ \omega(x): ~x \in \partial \Xi\}$ is uniformly equicontinuous.
\end{enumerate}
\end{defn}
We will denote the set of all compatible families of limit operators by $\mathfrak{lim}~\mathcal C_1(\mathcal H)$. 

Indeed, the properties of a compatible family of limit operators are the operator analogues of those in Lemma \ref{lemma:reconstr_boundary1} for functions. Then, the results of \cite{Fulsche2024} yield the following results for these compatible families of limit operators in the special case $\Xi = \mathbb R^{2n}$:
\begin{thm}\label{thm:limitops}
\begin{enumerate}[(1)]
\item If $B \in \mathcal C_1(\mathcal H)$, then $\omega(x) = \alpha_x(B), ~x \in \partial \Xi$ is a compatible family of limit operators.
\item If $\omega$ is a compatible family of limit operators, then there exists $B \in \mathcal C_1(\mathcal H)$ such that $\omega(x) = \alpha_x(B)$. $B$ is unique modulo $\mathcal K(\mathcal H)$.
\item Upon endowing the space $\mathfrak{lim}~ \mathcal C_1(\mathcal H)$ with pointwise addition, product, adjoint and the norm $\| \omega\| = \sup_{x \in \partial \Xi} \| \omega(x)\|_{op}$, it turns into a unital $C^\ast$-algebra.
\item The map $\mathcal C_1(\mathcal H) \ni B \mapsto [\omega(x) = \alpha_x(B)] \in \mathfrak{lim} ~\mathcal C_1(\mathcal H)$ is a surjective homomorphism of unital $C^\ast$-algebras. Its kernel is $\mathcal K(\mathcal H)$, and the quotient map $\mathcal C_1(\mathcal H)/\mathcal K(\mathcal H) \to \mathfrak{lim}~\mathcal C_1(\mathcal H)$ is an isomorphism of $C^\ast$-algebras.
\end{enumerate}
\end{thm}
We will spend the remaining part of this subsection on proving this result, which will be a central technical cornerstone for our later discussions, as well as future work. We first note that (1) of the theorem is straightforward to verify, so we are left with proving (2)-(4). The proofs of these facts are of course generalizations of the arguments from \cite[Section 6.3]{Fulsche2024}, and therefore bear rather obvious similarities.

As a first auxiliary fact, we need:
\begin{lem}
Let $B \in \mathcal C_1(\mathcal H)$. Then, $B \in \mathcal K(\mathcal H)$ if and only if $\alpha_x(B) = 0$ for every $x \in \partial \Xi$. 
\end{lem}
\begin{proof}
By Theorem \ref{thm:standardcorrespondences}, $B \in \mathcal K(\mathcal H)$ if and only if $A \ast B \in C_0(\Xi)$ for every $A \in \mathcal T^1(\mathcal H)$. Since $\alpha_x(A \ast B) = A \ast \alpha_x(B)$, and $A \ast \alpha_x(B) = 0$ for every $A \in \mathcal T^1(\mathcal H)$ if and only if $\alpha_x(B) = 0$, the statement follows.
\end{proof}
Hence, the map 
\begin{align*}
\mathcal C_1(\mathcal H) \ni B \mapsto (\alpha_x(B))_{x \in \partial\Xi}
\end{align*}
is a $^\ast$-homomorphism from the $C^\ast$-algebra $\mathcal C_1(\mathcal H)$ to the $C^\ast$-algebra 
\begin{align*}
\bigoplus_{x \in \partial\Xi}\mathcal C_1(\mathcal H),
\end{align*}
the kernel of which is $\mathcal K(\mathcal H)$. Here, the infinite direct sum on the right-hand side has to be understood as the $\ell^\infty$ sum. By general facts about $C^\ast$-algebras, the quotient map has to be an isometric isomorphism. In particular, we have:
\begin{lem}\label{lem:quotientnorm}
Let $B \in \mathcal C_1(\mathcal H)$. Then, the following equality holds: 
\begin{align*}
\| B + \mathcal K(\mathcal H)\| = \sup_{x \in \partial\Xi} \| \alpha_x(B)\|.
\end{align*}
\end{lem}
Note that $\mathfrak{lim}~\mathcal C_1(\mathcal H)$ can naturally be seen as a subset of $\bigoplus_{x \in \partial\Xi} \mathcal C_1(\mathcal H)$.

We will make use of the following standard construction: For any open neighborhood $O \subset \Xi$ of the identity $e \in \Xi$, we write $\varphi_O = \frac{1}{|O|}\chi_O$, where $|O|$ denotes the Haar measure of the set. When ordering the open neighborhoods of $e$ by inclusion, it is well-known that this turns the functions $\varphi_O$ into a bounded approximate identity of $L^1(\Xi)$. We will, for readability, write $(\varphi_\gamma)_{\gamma \in \Gamma}$ for this approximate identity. Fix any net $(b_\gamma)_{\gamma \in \Gamma} \subset (0, \infty)$, converging to $0$ and indexed by the same directed set (e.g., let $0 \neq f \in L^1(\Xi)$ and set $b_\gamma = \| f - f \ast \varphi_\gamma\|_{L^1}$). Fix a regular set $S \subset \mathcal T^1(\mathcal H)$. Then, since $\{ A \ast A: ~A \in S\}$ is a regular subset of $L^1(\Xi)$, we can find for every $\gamma \in \Gamma$ a finite collection of $c_j^\gamma \in \mathbb C$, $x_j^\gamma \in \Xi$ and $A_j^\gamma \in S$ such that:
\begin{align*}
\| \varphi_\gamma - \sum_j c_j^\gamma \alpha_{x_j^\gamma}(A_j^\gamma \ast A_j^\gamma)\|_{L^1} < b_\gamma.
\end{align*}
All this data will be considered fixed for the remainder of this section.
\begin{lem}\label{lemma:uniform_approximable}
Let $T \subset \mathcal C_1(\mathcal H)$ be a norm-bounded subset which is uniformly equicontinuous, i.e., the functions $\Xi \ni x \mapsto \alpha_x(B), ~B \in T$, are uniformly equicontinuous. Then, 
\begin{align*}
\sup_{B \in T} \| B - \sum_j \left( c_j^\gamma \alpha_{x_j^\gamma}(A_j^\gamma \ast A_j^\gamma) \ast B \right) \|_{op} \overset{\gamma \in \Gamma}{\longrightarrow} 0.
\end{align*}
\end{lem}
\begin{proof}
We estimate:
\begin{align*}
\| B - \sum_j & \left( c_j^\gamma \alpha_{x_j^\gamma}(A_j^\gamma \ast A_j^\gamma) \ast B \right) \|_{op} \\
&\leq \| B - \varphi_\gamma \ast B\|_{op} + \| \varphi_\gamma \ast B - \sum_j \left( c_j^\gamma \alpha_{x_j^\gamma}(A_j^\gamma \ast A_j^\gamma) \ast B \right) \|_{op}\\
&\leq \| B - \varphi_\gamma \ast B\|_{op} + \| \varphi_\gamma - \sum_j \left( c_j^\gamma \alpha_{x_j^\gamma}(A_j^\gamma \ast A_j^\gamma)  \right) \|_{L^1} \| B\|_{op}\\
&\leq \| B - \varphi_\gamma \ast B\|_{op} + \sup_{C \in T} \| C\| b_\gamma
\end{align*}
Let $\varepsilon > 0$. Since the $B \in T$ are uniformly equicontinuous, there exists a neighborhood $O \subset \Xi$ of $e \in \Xi$ such that for $x \in O$: $\| B - \alpha_x(B)\| < \varepsilon$. Then,
\begin{align*}
\| B - \varphi_\gamma \ast B\|_{op} &\leq \int_\Xi \varphi_\gamma(x) \| B - \alpha_x(B)\|_{op}~dx\\
&\leq \varepsilon \int_O \varphi_\gamma(x)~dx + 2\| B\|_{op} \int_{O^c} \varphi_\gamma(x)~dx.
\end{align*}
By definition of the $\varphi_\gamma$, there exists $\gamma_0 \in \Gamma$ such that for $\gamma \geq \gamma_0$: $\int_{O^c} \varphi_\gamma(x)~dx = 0$. Thus, for $\gamma \geq \gamma_0$ we have:
\begin{align*}
\| B - \sum_j & \left( c_j^\gamma \alpha_{x_j^\gamma}(A_j^\gamma \ast A_j^\gamma) \ast B \right) \|_{op} \leq \varepsilon + \sup_{C \in T} \| C\| b_\gamma.
\end{align*}
Since $\varepsilon > 0$ was arbitrary and $b_\gamma \to 0$, the statement follows.
\end{proof}
\begin{proof}[Proof of Theorem \ref{thm:limitops}]
As already mentioned, we omit the proof of (1), which is straightforward. We already know that the map $\Phi(B)(x) = \alpha_x(B)$ maps $\Phi: \mathcal C_1(\mathcal H)/\mathcal K(\mathcal H) \to \mathfrak{lim}\mathcal C_1(\mathcal H)$ injectively. Hence, (3) and (4) follow immediately once (2) is proven.

Therefore, fix some $\omega \in \mathfrak{lim}~\mathcal C_1(\mathcal H)$. Further, we let the approximate identity $(\varphi_\gamma)_{\gamma \in \Gamma}$, the regular set $S \subset \mathcal T^1(\mathcal H)$, the net $(b_\gamma)_{\gamma \in \Gamma} \subset (0, \infty)$ and the $c_j^\gamma, x_j^\gamma, A_j^\gamma$ as earlier. For $A \in S$ write $f_A(x) = (\omega(x) \ast A)(e) = \tr(\beta_-(\omega(x))A)$. Then, $f_A \in C(\partial \Xi)$ such that by Tietze's extension theorem there exists $f_{A, 0} \in \BUC$ such that $\alpha_x(f_{A, 0}) = f_A(x)$ for every $x \in \partial\Xi$. By Lemma \ref{lemma:uniform_approximable}, we know that 
\begin{align*}
&\sup_{x \in \partial \Xi} \| \omega(x) - \sum_j c_j^\gamma \alpha_{x_j^\gamma} (A_j^\gamma \ast \alpha_x(\beta_-(f_{A_j^\gamma,0})))\|_{op}\\
&= \sup_{x \in \partial \Xi} \| \omega(x) - \sum_j c_j^\gamma \alpha_{x_j^\gamma} (A_j^\gamma \ast A_j^\gamma \ast \omega(x))\|_{op} \overset{\gamma \in \Gamma}{\longrightarrow} 0.
\end{align*}
By Lemma \ref{lem:quotientnorm}, $\sum_{j} c_j^\gamma \alpha_{x_j^\gamma}(A_j^\gamma \ast \beta_-(f_{A_j^\gamma, 0})) + \mathcal K(\mathcal H)$ is a Cauchy net in $\mathcal C_1(\mathcal H)/\mathcal K(\mathcal H)$ such that there exists $B \in \mathcal C_1(\mathcal H)$ with:
\begin{align*}
\sum_{j} c_j^\gamma \alpha_{x_j^\gamma}(A_j^\gamma \ast \beta_-(f_{A_j^\gamma, 0})) + \mathcal K(\mathcal H) \overset{\gamma \in \Gamma}{\longrightarrow} B + \mathcal K(\mathcal H).
\end{align*}
Now, for $x \in \partial \Xi$:
\begin{align*}
\alpha_x(B) = \lim_{\gamma \in \Gamma} \sum_j c_j^\gamma \alpha_{x_j^\gamma}(A_j^\gamma \ast \alpha_x(\beta_-(f_{A_j^\gamma, 0}))) = \omega(x),
\end{align*}
which finishes the proof.
\end{proof}

\subsection{Wiener's Tauberian theorems for QHA}

Let now again $\mathcal A$ be an $\alpha$-invariant $C^\ast$-subalgebra of $\operatorname{BUC}(\Xi)$ containing $C_0(\Xi)$ and $\tau'$ a Hausdorff vector topology of $\mathcal L(\mathcal H)$, which is finer than $\tau_{w^\ast}'$. We say that the limit operators of $A$ exist on $\partial_{\beta_-(\mathcal A)}(G)$ with respect to $\tau'$ if for every $x \in \partial_{\beta_-(\mathcal A)}(G)$ and every net $(x_\gamma)\subset \Xi$ with $x_\gamma \to x$ the limit $\lim_{\gamma \in \Gamma} \alpha_{x_\gamma}(A)$ exists with respect to the topology $\tau'$. We make the following definitions:\nomenclature{$\mathcal A_{j,op}(\tau')$}{cf.\ p.\ \pageref{def:Ajop}}\label{def:Ajop}
\begin{align*}
\mathcal A_{3,op}(\tau') &= \{ B \in \mathcal L(\mathcal H): ~\text{the limit operators of $B$ exist on $\partial_{\beta_-(\mathcal A)}(G)$ with respect to $\tau'$}\}\\
\mathcal A_{2, op}(\tau') &= \{ B \in \mathcal A_{3,op}(\tau'): ~\alpha_x(B) \in \mathcal C_1(\mathcal H) \text{ for every } x \in \partial_{\beta_-(\mathcal A)}(\Xi) \}\\
\mathcal A_{1,op}(\tau') &= \{ B \in \mathcal A_{2,op}(\tau'): ~\{ \alpha_x(B): ~x\in \partial_{\beta_-(\mathcal A)}(\Xi) \} \text{ unif.\ equicont.}\}.
\end{align*}
Further, we let $\mathcal A_{0, op} = \mathcal A_{0,op}(\tau')$ denote the space corresponding to $\mathcal A$ in the sense of the Correspondence Theorem \ref{theorem:correspondence}, which is always a $C^\ast$-subalgebra of $\mathcal C_1(\mathcal H)$ by Theorem \ref{thm:algebra}. Again, this is independent of the chosen topology $\tau'$. We will write $\operatorname{BUC}_{j,op}(\tau')$ for $\mathcal A_{j,op}(\tau')$ in the case where $\mathcal A = \BUC$. Note that $\operatorname{BUC}_{3, op}(\tau_{w^\ast}') = \mathcal L(\mathcal H)$.
\begin{rem}
    We chose to add an apostrophe to the notion of the topologies here to clearly distinguish between topologies on spaces of functions and spaces of operators: Notions such as $\tau$ or $\tau_{w^\ast}$ will always refer to a topology on a space of functions, while $\tau'$ or $\tau_{w^\ast}'$ will refer to topologies on spaces of operators. In later stages of the paper, making this notational distinction explicit is useful to avoid ambiguities.
\end{rem}

We now obtain the following result, which is the operator analogue of Lemma \ref{lem:mainthm1}. The proof is essentially the same as for its classical version, making use of the same topological lemma, and is hence omitted.
\begin{lem}\label{lem:mainthm3}
    Let $\tau'$ and $\mathcal A$ as above and $j \in\{ 1, 2, 3\}$. Then, the following holds true:
    \begin{align*}
        \mathcal A_{j,op}(\tau') = \mathrm{BUC}(G)_{j,op}(\tau_{w^\ast}') \cap \mathcal A_{3,op}(\tau_{w^\ast}') \cap \mathrm{BUC}(G)_{3,op}(\tau')
    \end{align*}
\end{lem}

We now obtain the following result:
\begin{thm}\label{thm:mainthm3}
Let $S \subset L^1(\Xi)$ be a regular subset, $\tau'$ and $\mathcal A$ as above, $j \in \{ 0, 1, 2, 3\}$ and $B \in \operatorname{BUC}_{j, op}(\tau')$. Then, the following holds true:
\begin{align*}
B \in \mathcal A_{j, op}(\tau') \Leftrightarrow g \ast B \in \mathcal A_{0,op} \text{ for every } g \in S.
\end{align*}
\end{thm}
\begin{proof}
As in the case of functions, also justified by the previous lemma, it suffices to consider the case where $j = 3$ and $\tau' = \tau_{w^\ast}'$. By essentially the same facts as in the proof of Theorem \ref{mainthm:1}, $B \in \mathcal A_{3, op}(\tau_{w^\ast}')$ if and only if $\alpha_{y_1}(B) = \alpha_{y_2}(B)$ for every $y_1, y_2 \in \varphi^{-1}(x)$, where $x \in \partial \Xi$ and $\varphi$ is as in that proof. If $B \in \mathcal A_{3, op}(\tau_{w^\ast}')$, then for $A \in \mathcal T^1(\mathcal H)$ we easily obtain that $A \ast B \in \mathcal A$. Hence, we also have $g \ast A \ast B = A \ast (g \ast B) \in \mathcal A$. Since $g \ast B \in \mathcal C_1(\mathcal H)$, the correspondence theorem yields $g \ast B \in \mathcal A_{0, op}$.

On the other hand, let us assume $g \ast B \in \mathcal A_{0, op}(\tau')$ for every $g \in S$. Then, it is for $y_1, y_2 \in \varphi^{-1}(x)$ and $g \in S$:
\begin{align*}
g \ast \alpha_{y_1}(B) = \alpha_{y_1}(g \ast B) = \alpha_{y_2}(g \ast B) = g \ast \alpha_{y_2}(B).
\end{align*}
By Wiener's approximation theorem, this extends to every $g \in L^1(\Xi)$. In particular, letting $\varphi_\gamma$ be an approximate identity in $L^1(\Xi)$, we obtain in weak$^\ast$ topology:
\begin{align*}
\varphi_\gamma \ast \alpha_{y_1}(B) \to \alpha_{y_1}(B).
\end{align*}
But the left-hand side converges also to $\alpha_{y_2}(B)$. Hence, we obtain $B \in \mathcal A_{3, op}(\tau_{w^\ast}')$.
\end{proof}
For $\mathcal I$ a closed, $\alpha$-invariant ideal of $\mathcal A$, we now set\label{def:Ijop}
\begin{align*}\nomenclature{$\mathcal I_{j,op}(\tau')$}{cf.\ p.\ \pageref{def:Ijop}}
\mathcal I_{j,op}(\tau') = \{ B \in \mathcal A_{j, op}(\tau'): ~\alpha_x(B) = 0 \text{ whenever } x \in I\}.
\end{align*}
Note that $\mathcal I_{0,op} = \mathcal I_{0,op}(\tau')$ is again independent of the particular topology $\tau'$. Further, it is the space corresponding to $\mathcal I$ in the sense of the Correspondence Theorem \ref{theorem:correspondence}, and it is a closed two-sided ideal in $\mathcal A_{0, op}$ by Theorem \ref{thm:algebra}.
\begin{thm}\label{thm:mainthm4}
Let $\mathcal A, \mathcal I$ and $\tau'$ as above, $j \in \{ 0, 1, 2, 3\}$. Further, let $S \subset L^1(\Xi)$ be a regular subset and $C \in \mathcal A_{j, op}(\tau')$. Then, for $B \in \mathcal A_{j, op}(\tau')$ the following are equivalent:
\begin{enumerate}[(i)]
\item $B - C \in \mathcal I_{j,op}(\tau')$;
\item $\alpha_x(B) = \alpha_x(C)$ for every $x \in I$;
\item $g \ast B - g \ast C \in \mathcal I_{0,op}$ for every $g \in S$;
\item $g \ast B - g \ast C \in \mathcal I_{0,op}$ for every $g \in L^1(\Xi)$.
\end{enumerate}
\end{thm}
\begin{proof}
The proof is now entirely analogous to that of Theorem \ref{mainthm:2}.
\end{proof}
We now establish versions of these theorems for regular sets of operators.
\begin{thm}\label{thm:Tauberian3}
Let $S \subset \mathcal T^1(\mathcal H)$ be a regular subset, $\tau$ and $\mathcal A$ as above, $j \in \{ 0, 1, 2, 3\}$ and $f \in \operatorname{BUC}_j(\tau)$. Then, the following holds:
\begin{align*}
f \in \mathcal A_j(\tau) \Longleftrightarrow A \ast f \in \mathcal A_{0,op} \text{ for every }  A \in S.
\end{align*}
\end{thm}
\begin{proof}
Using Theorem \ref{mainthm:1} and the Correspondence Theorem, one has
\begin{align*}
f \in \mathcal A_j(\tau) &\Leftrightarrow g \ast f \in \mathcal A \text{ for every } g \in L^1(\Xi)\\
&\Leftrightarrow g \ast A \ast f \in \mathcal A_{0,op} \text{ for every } A \in \mathcal S \text{ and } g \in L^1(\Xi)\\
&\Leftrightarrow A \ast f \in \mathcal A_{0,op} \text{ for every }  A \in \mathcal S.\qedhere
\end{align*}
\end{proof}
\begin{thm}\label{thm:Tauberian4}
Let $S \subset \mathcal T^1(\mathcal H)$ be a regular subset, $\tau'$ and $\mathcal A$ as above, $j \in \{ 0, 1, 2, 3\}$ and $B \in \operatorname{BUC}_{j,op}(\tau')$. Then, the following holds:
\begin{align*}
B \in \mathcal A_{j,op}(\tau') \Leftrightarrow A \ast B \in \mathcal A \text{ for every } A \in S.
\end{align*}
\end{thm}
\begin{proof}
By Theorem \ref{thm:mainthm3} and the Correspondence Theorem it is:
\begin{align*}
A \in \mathcal A_{j,op}(\tau') &\Leftrightarrow g \ast B \in \mathcal A_{0,op} \text{ for every } g \in L^1(\Xi)\\
&\Leftrightarrow g \ast A \ast B \in \mathcal A \text{ for every } A \in S \text{ and } g \in L^1(\Xi)\\
&\Leftrightarrow A \ast B \in \mathcal A \text{ for every } A \in S.\qedhere
\end{align*}
\end{proof}
\begin{rem}
    The above two results show that, in the language of \emph{corresponding spaces} (cf.\ \cite[Section IV]{werner84}), $\mathcal A_j(\tau_{w^\ast})$ and $\mathcal A_{j,op}(\tau_{w^\ast}')$ are corresponding spaces, and in some sense they are the maximal useful extension of the corresponding pair $(\mathcal A, \mathcal A_{0, op})$.
\end{rem}
We list two more results, which are in the same spirit:
\begin{thm}\label{thm:Tauberian1}
Let $\mathcal A, \mathcal I$ and $\tau$ as above, $j \in \{ 0, 1, 2, 3\}$. Further, let $S \subset \mathcal T^1(\mathcal H)$ be a regular subset and $a \in \mathcal A_{j}(\tau)$. Then, for $f \in  \mathcal A_j(\tau)$ the following are equivalent:
\begin{enumerate}[(i)]
\item $f - a \in \mathcal I_j(\tau)$;
\item $A \ast f - A \ast a \in \mathcal I_{0,op}$ for every $A \in S$;
\item $A \ast f - A \ast a \in \mathcal I_{0,op}$ for every $A \in \mathcal T^1(\mathcal H)$.
\end{enumerate}
\end{thm}

\begin{thm}\label{thm:Tauberian2}
Let $\mathcal A, \mathcal I$ and $\tau'$ as above, $j \in \{ 0, 1, 2, 3\}$. Further, let $S \subset \mathcal T^1(\mathcal H)$ be a regular subset and $C \in \mathcal A_{j,op}(\tau')$. Then, for $B \in  \mathcal A_{j,op}(\tau')$ the following are equivalent:
\begin{enumerate}[(i)]
\item $B - C \in \mathcal I_{j,op}(\tau')$;
\item $A \ast B - A \ast C \in \mathcal I$ for every $A \in S$;
\item $A \ast B - A \ast C \in \mathcal I$ for every $A \in \mathcal T^1(\mathcal H)$.
\end{enumerate}
\end{thm}
\begin{proof}[Proofs of Theorems \ref{thm:Tauberian1}-\ref{thm:Tauberian2}]
Similarly to Theorems \ref{thm:Tauberian3} and \ref{thm:Tauberian4}, these results are immediate consequences of Theorems \ref{mainthm:2} and \ref{thm:mainthm4} together with the Correspondence Theorem.
\end{proof}
\begin{rem}
Letting in these theorems $\mathcal A = \operatorname{BUC}(\Xi)$, $\mathcal I = C_0(\Xi)$, $j = 3$, $\tau = \tau_{w^\ast}$, $\tau' = \tau_{w^\ast}'$, $g = const$ and $B = const \cdot Id$ reproduces the Wiener Tauberian Theorems for operators from \cite[Theorem 5.1]{Luef_Skrettingland2021} in the case where $\Xi = \mathbb R^{2n}$.
\end{rem}

One might consider the topology $\tau_{s^\ast}'$ as the analog of $\tau_{u.c.}$ on operators. In analogy to the function case, we therefore let\nomenclature{$\operatorname{SO}(\mathcal H)$}{Slowly oscillating operators, p.\ \pageref{def:SOH}}\label{def:SOH}\nomenclature{$B_0(\mathcal H)$}{Bounded operators vanishing at infinity, p.\ \pageref{def:B0H}}\label{def:B0H}
\begin{align*}
\operatorname{SO}(\mathcal H) := \operatorname{BUC}(\Xi)_{1,op}(\tau_{s^\ast}'),\\
B_0(\mathcal H) := C_0(\Xi)_{1,op}(\tau_{s^\ast}').
\end{align*}
Using Theorem \ref{thm:limitops}.(2) instead of Lemma \ref{lemma:reconstr_boundary1}, it is now easy to prove that
\begin{align*}
\operatorname{SO}(\mathcal H) = \mathcal C_1(\mathcal H) + B_0(\mathcal H).
\end{align*}
The operator analog of Pitt's extension can be seen as the following statement, which is strictly more general than \cite[Theorem 5.2]{Luef_Skrettingland2021}:
\begin{cor}\label{pitt}
Let $S_0 \subset L^1(\Xi)$ and $S_1 \subset \mathcal T^1(\mathcal H)$ be regular subsets and $B \in \operatorname{SO}(\mathcal H)$. Then, the following are equivalent:
\begin{enumerate}[(i)]
\item $B \in B_0(\mathcal H)$.
\item $A \ast B \in C_0(\Xi)$ for every $A \in S_1$.
\item $g \ast B \in \mathcal K(\mathcal H)$ for every $g \in S_0$.
\end{enumerate}
If $f \in \operatorname{SO}(\Xi)$, then the following statements are equivalent:
\begin{enumerate}[(i)]
\item $f \in B_0(\Xi)$.
\item $A \ast f \in \mathcal K(\mathcal H)$ for every $A \in S_1$.
\item $g \ast f \in C_0(\Xi)$ for every $g \in S_0$.
\end{enumerate}
\end{cor}
To bring the discussion regarding the operator version of Pitt's extension to a pleasant ending, we provide the following characterization of $\operatorname{SO}(\mathcal H)$, which is rather close to the initial definition of $\operatorname{SO}(\Xi)$. In contrast to the case of functions, the placement of the quantifiers is a little different:
\begin{prop}
    The space $\operatorname{SO}(\mathcal H)$ agrees with the set of all $A \in \mathcal L(\mathcal H)$ such that $B \in \{ A, A^\ast\}$ satisfy the following: For every $\varepsilon > 0$ there exists a neighborhood $O_\varepsilon \subset \Xi$ of $e$ such that for each $g \in \mathcal H$ with $\| g\| = 1$ there exists a compact set $K_{g, \varepsilon} \subset \Xi$ with the following property:
    \begin{align*}
        \forall y \in O_\varepsilon, z \not \in K_{g, \varepsilon}: ~\| \alpha_z(B - \alpha_y(B)) g\| < \varepsilon.
    \end{align*}
\end{prop}
\begin{proof}
    We first note that the above property holds for $A \in \mathcal C_1(\mathcal H)$. Further, for $A \in B_0(\mathcal H)$ we have
    \begin{align*}
        \| \alpha_{z_\gamma}(A)g\| \to 0,
    \end{align*}
    as $z_\gamma \to x \in \partial \Xi$. The same is true (by considering the embedding $C_0(\Xi) \oplus \mathbb C 1 \hookrightarrow \operatorname{BUC}(\Xi)$ and the surjective map $\sigma \Xi \to \alpha \Xi$, where $\alpha \Xi$ is the one-point compactification, as done earlier) in $\alpha \Xi$ (which is exactly the maximal ideal space of $C_0(\Xi) \oplus \mathbb C 1$), where the neighborhoods of the point at infinity are the complements of compact sets. Hence for a given $\varepsilon > 0$ there exists $K \subset \Xi$ compact such that we have $\| \alpha_{z}(A)g\| < \varepsilon$ once $z \not \in K$. Therefore, for a neighborhood $O$ of $e$ small enough, we have
    \begin{align*}
        \| \alpha_z(A - \alpha_y(A))g\| \leq \| \alpha_z(A)g\| + \| \alpha_{z+y}(A)g\| \leq 2\varepsilon,
    \end{align*}
    whenever $z$, $z + y \not \in K$. By the definition of $\operatorname{SO}(\mathcal H)$, the above two points show that $\operatorname{SO}(\mathcal H)$ is included in the space described above. 

    Now we assume that the operator $A$ satisfies the above property. We will show that then $A \in \BUC_{1, op}(\tau_S')$, where $\tau_S'$ is the strong operator topology. Since $A^\ast$ also satisfies the described properties, this then implies that $A \in \BUC_{1, op}(\tau_{s^\ast}')$.

    First, we show that $\alpha_{x_\gamma}(A) \to \alpha_x(A)$ in strong operator topology when $(x_\gamma) \subset \Xi$, $x_\gamma \to x \in \partial \Xi$. For this, let $\varepsilon > 0$, $O_\varepsilon \subset \Xi$, $g \in \mathcal H$ with $\| g\| = 1$ and $K \subset \Xi$ as in the condition on $A$. Further, as in the proof of Proposition \ref{decomp:SO}, let $\psi_\gamma$ be the same approximate identity of $L^1(\Xi)$ (i.e., $\psi_\gamma = \mathbf \chi_{O_\gamma}/|O_\gamma|$, where the $O_\gamma$ are a neighborhood basis of the identity element, ordered by inclusion). By a standard argument for nets, there is no loss of generality in assuming that both $\psi_\gamma$ and $x_\gamma$ are indexed by the same directed set $\Gamma$. Then, 
    \begin{align*}
        \| \alpha_{x_\gamma}(A)(g) - &\psi_{\gamma_0} \ast \alpha_{x_\gamma}(A)(g)\| \\
        &\leq \int_\Xi \| \alpha_{x_\gamma}(A)(g) - \alpha_{x_\gamma + z}(A)(g)\| \psi_{\gamma_0}(z)~dz\\
        &\leq \varepsilon,
    \end{align*}
    provided $\gamma > \gamma_0$ such that $O_{\gamma_0} \subset O_\varepsilon$ and $x_\gamma \not \in K$. Therefore, for $\gamma, \gamma' > \gamma_0$:
    \begin{align*}
        \| \alpha_{x_\gamma}(A)(g) - \alpha_{x_{\gamma'}}(A)(g)\| \leq 2\varepsilon + \| \psi_{\gamma_0} \ast \alpha_{x_\gamma}(A)(g) - \psi_{\gamma_0} \ast \alpha_{x_{\gamma'}}(A)(g)\|.
    \end{align*}
    Since $\psi_{\gamma_0} \ast \alpha_{x_\gamma}(A) = \alpha_{x_\gamma}(\psi_{\gamma_0} \ast A)$ and $\psi_{\gamma_0} \ast A \in \mathcal C_1(\mathcal H)$, we know that $\alpha_{x_\gamma}(\psi_{\gamma_0} \ast A) \to \alpha_x(\psi_{\gamma_0} \ast A)$ in strong operator topology. Hence, $\alpha_{x_\gamma}(A)(g)$ is a Cauchy net, therefore $\alpha_{x_\gamma}(A)$ converges in SOT. Of course, this SOT-limit has to agree with the WOT limit of the net, hence $\alpha_{x_\gamma}(A)(g) \to \alpha_x(A)(g)$ in SOT. 

    Next, we come to the uniform equicontinuity of the $\alpha_{x}(A)$. Note in the previous part of the proof that the choice of the neighborhood $O_\varepsilon$ did not depend on the limit point $x \in \partial \Xi$. Hence, as in the proof of Proposition \ref{decomp:SO}, one obtains
    \begin{align*}
        \| \alpha_x(A)(g) - \alpha_{x-y}(A)(g)\| \leq \varepsilon
    \end{align*}
    for each $x \in \partial \Xi$, $g\in \mathcal H$ with $\| g\| = 1$ and $y \in O_\varepsilon$. This clearly implies that the family of limit operators is uniformly equicontinuous. Hence, $A \in \BUC_{1, op}(\tau_S')$ which finishes the proof. 
\end{proof}

As mentioned above, the function space $B_0(\Xi)$ consists exactly of those functions $f \in L^\infty(\Xi)$ such that the modulations $\gamma_\xi(f) = m((\cdot), \xi)f$ satisfy 
\begin{align*}
\Xi \ni \xi \mapsto \gamma_\xi(f) \quad \text{ is } \| \cdot\|_\infty-\text{continuous}.
\end{align*}
If the phase space $(\Xi, m)$ is 2-regular (cf.\ \cite{Fulsche_Galke2023}, Example 2.7 and Lemma 6.15), then the natural modulation of an operator $B \in \mathcal L(\mathcal H)$ is given by:
\begin{align*}
\gamma_\xi(B) = U_{-\xi/2} B U_{-\xi/2}.
\end{align*}
Letting now $R$ be the parity operator on $\mathcal H$, one easily sees that $R\gamma_w(A) = \alpha_{w/2}(RA)$. In particular, the map $A \mapsto RA$ is a bijective linear map between $\mathcal C_1(\mathcal H)$ and
\begin{align*}
\mathcal E(\mathcal H) := \{ B \in \mathcal L(\mathcal H): ~\Xi \ni \xi \mapsto \gamma_\xi(B) \text{ is } \| \cdot\|_{op}\text{-cont.}\}.
\end{align*}
It is not difficult to verify that $\mathcal E(\mathcal H)$ is a closed subspace of $\mathcal L(\mathcal H)$, being invariant under the adjoint map. Further, it is a left- and right-$\mathcal C_1(\mathcal H)$ module. Unfortunately, we have $\mathcal E(\mathcal H) \neq B_0(\mathcal H)$:  To see this, note that $R \in \mathcal E(\mathcal H)$ (because $Id \in \mathcal C_1(\mathcal H)$). We clearly have $\| \alpha_w(R)f\| = \| W_{2w} Rf\| = \| f\|$ for every $f \in \mathcal H$, so $\alpha_w(R)$ cannot converge to $0$ in the strong operator topology. However, it is true that $\alpha_x(B) = 0$ for every $B \in \mathcal E(\mathcal H)$ and $x \in \partial \Xi$: This is because
\begin{align*}
\alpha_x(B) = \alpha_x(RA) = \alpha_x(R) \alpha_x(A)
\end{align*}
for some $A \in \mathcal C_1(\mathcal H)$ and $\alpha_x(R) = 0$ for every $x \in \partial \Xi$. To fix this problem, one could consider the following larger spaces:
\begin{align*}
\operatorname{SO}_w(\mathcal H) = \operatorname{BUC}(\Xi)_{1,op}(\tau_{w^\ast}')\\
B_{0,w}(\mathcal H) = C_0(\Xi)_{1,op}(\tau_{w^\ast}').
\end{align*}
Then, employing again Theorem \ref{thm:limitops}.(2) it is
\begin{align*}
\operatorname{SO}_w(\mathcal H) = \mathcal C_1(\mathcal H) + B_{0,w}(\mathcal H)
\end{align*}
and the analogous operator version for Pitt's extension, Corollary \ref{pitt}, still holds. In addition, we have $\mathcal E(\mathcal H) \subseteq  B_{0,w}(\mathcal H)$. 

So far, it might very well be that $B_0(\mathcal H)$ and $\mathcal K(\mathcal H)$ are indeed the same spaces. We will finish by giving examples,  which show that we indeed have $\mathcal K(\mathcal H) \subsetneq B_0(\mathcal H)$. We will first give an example for the phase space $\mathbb Z \times \mathbb T$, i.e.\ an operator on $\ell^2(\mathbb Z)$. Afterwards, we will provide a modification of this example for the phase space $\mathbb R^2$, i.e.\ on $L^2(\mathbb R)$.

Similarly, one might ask the same questions about $\mathcal E(\mathcal H)$ and $B_{0,w}(\mathcal H)$. Modifications of essentially the same examples will show that $\mathcal E(\mathcal H) \subsetneq B_{0,w}(\mathcal H)$.

\subsection{Example on $\Xi = \mathbb Z \times \mathbb T$}

We consider the phase space $\Xi = \mathbb Z \times \mathbb T$, endowed with the multiplier $m((k, \theta), (m, \zeta)) = \zeta^{-k}$. Note that, in contrast to all previous conventions, we will still write the group operation in $\mathbb T = \{ \theta \in \mathbb C: ~|\theta| = 1\}$ multiplicatively. 

The representation of the phase space is given on $\mathcal H = \ell^2(\mathbb Z)$ by $U_{(k, \theta)}a(n) = \theta^n a(n-k)$. With this data, we indeed have $B_0(\mathcal H) \supsetneq \mathcal K(\mathcal H)$. To see this, we first note the following, where we will make use of the fact that we can write each $A \in \mathcal L(\ell^2(\mathbb Z))$ in its infinite matrix form $A = (a_{j,k})_{j,k \in \mathbb Z}$ with respect to the standard basis.
\begin{prop}\label{prop:char_SOT_to_limit}
    Let $A \in \mathcal L(\ell^2(\mathbb Z))$. Then, $\alpha_{(k, \vartheta)}(A) \to 0$ in SOT as $(k, \vartheta) \to \partial \Xi$ if and only if the $\ell^2$-norms of the columns of $A$ converge to zero, that is:
    \begin{align*}
        \sum_{j \in \mathbb Z} |a_{j,k}|^2 \to 0 \quad \text{ as } k \to \pm \infty.
    \end{align*}
\end{prop}
\begin{proof}
    Assume that $\alpha_{(k, \vartheta)}(A)\to 0$ in SOT as $(k, \vartheta) \to \partial \Xi$. Then, we have:
    \begin{align*}
        \| A U_{(-k, \vartheta^{-1})} e_0\|^2 = \| Ae_{-k}\|^2 = \sum_{j \in \mathbb Z} |a_{j,-k}|^2.
    \end{align*}
    Since $\widehat{\mathbb Z} = \mathbb T$ is compact, for $(k_\gamma, \vartheta_\gamma) \to \partial (\mathbb Z \times \mathbb T)$ we necessarily have $|k_\gamma| \to \infty$. Since we are interested in the limit operators all being zero, we can consider the maximal ideal space of $C_0(\mathbb Z \times \mathbb T) \oplus \mathbb C 1$ instead of $\BUC$, i.e., we can work in the one-point compactification of $\Xi$ instead. Hence, instead of working with a general net $k_\gamma$ converging to $\pm \infty$, we can consider the limit $k \to \pm \infty$. Therefore, $\alpha_{(k, \vartheta)}(A) \to 0$ in SOT as $k \to \pm \infty$ implies
    \begin{align*}
        \sum_{j \in \mathbb Z} |a_{j,-k}|^2 \to  0, \quad k \to \pm \infty. 
    \end{align*}
    On the other hand, if we assume that the columns of $A$ converge to zero in the above sense, then this implies (by essentially the same reasoning as above) that 
    \begin{align*}
        \| AU_{(-k, \vartheta^{-1})} e_m\|^2 \to 0, \quad k \to \pm \infty
    \end{align*}
    for each $m \in \mathbb Z$. Therefore, we also obtain that for every compactly supported $b \in \ell^2(\mathbb Z)$ we have $\alpha_{(k, \vartheta)}(A)b \to 0$ in $\ell^2(\mathbb Z)$, and hence by density we see that $\alpha_{(k, \vartheta)}(A) \to 0$ in SOT when $k \to \pm \infty$. 
\end{proof}
\begin{cor}
    Let $A \in \mathcal L(\ell^2(\mathbb Z))$. Then, $A \in B_0(\ell^2(\mathbb Z))$ if and only if
    \begin{align*}
        \sum_{j \in \mathbb Z} |a_{j,k}|^2 \to 0 \quad \text{ as } k \to \pm \infty
    \end{align*}
    and
    \begin{align*}
        \sum_{k \in \mathbb Z} |a_{j,k}|^2 \to 0 \quad \text{ as } j \to \pm \infty.
    \end{align*}
\end{cor}
\begin{proof}
    For $A \in B_0(\ell^2(\mathbb Z))$ we have that $\alpha_{(k, \vartheta)}(B) \to 0$ in SOT as $(k, \vartheta) \to \partial \Xi$ for both $B = A$ and $B = A^\ast$. Applying the previous proposition shows the equivalence.
\end{proof}
\begin{prop}
    We have $\mathcal K(\ell^2(\mathbb Z)) \subsetneq B_0(\ell^2(\mathbb Z))$.
\end{prop}
\begin{proof}
    The example we will use is essentially the example taken from Problem 177 in Halmos' classical Hilbert space problem book \cite{Halmos1982}. We let $P$ be the orthogonal projection from $\ell^2(\mathbb Z)$ to $\ell^2(\mathbb N_0)$ and consider the operator $A = A_0 P$, where $A_0$ is the operator on $\ell^2(\mathbb N_0)$ given by the infinite matrix
    \begin{align*}
        A_0 = \begin{pmatrix}
           1 & 0 & 0 & 0 & 0 & 0 & \\
           0 & \frac 12 & \frac 12 & 0 & 0 & 0 & \\
           0 & \frac 12 & \frac 12 & 0 & 0 & 0 & \\
           0 & 0 & 0 & \frac 13 & \frac 13 & \frac 13 & \\
           0 & 0 & 0 & \frac 13 & \frac 13 & \frac 13 & \\
           0 & 0 & 0 & \frac 13 & \frac 13 & \frac 13 & \\
           & & & & & & \ddots
        \end{pmatrix}
    \end{align*}
    Then, the $\ell^2$-norms of the columns and of the rows of $A$ are clearly of the form $\sqrt{n\frac{1}{n^2}} = \frac{1}{\sqrt{n}}$, hence converge to 0 for $k \to -\infty$ and $j \to -\infty$, respectively (and are constantly zero for $k > 0$ and $j > 0$ anyway), hence $A \in B_0(\ell^2(\mathbb Z))$. Since $A_0$ consists of the direct sum of orthogonal projections, it is clearly bounded but not compact.
\end{proof}

\begin{prop}
    We have $\mathcal E(\ell^2(\mathbb Z)) \subsetneq B_{0,w}(\ell^2(\mathbb Z))$.
\end{prop}
\begin{proof}
    Let $A_0$ and $P$ be as in the proof of the previous proposition. Since $A_0 P \not \in \mathcal C_1(\ell^2(\mathbb Z))$, we also have $R A_0 P \not \in \mathcal E(\ell^2(\mathbb Z))$. An easy application of Proposition \ref{prop:char_SOT_to_limit} nevertheless shows that $\alpha_{(k, \theta)}(RA_0 P)$ converges even in SOT to the limit operators, which are all zero. Hence, $RA_0 P \in B_{0, w}(\ell^2(\mathbb Z))$.
\end{proof}

\subsection{Example on $\Xi = \mathbb R \times \mathbb R$}
We now deal with the phase space $\Xi = \mathbb R \times \mathbb R$. Here, the multiplier is $m((x, \xi), (y, \eta)) = e^{-ix\eta}$ and the representation acts on $\mathcal H = L^2(\mathbb R)$ and is given by $U_{(x, \xi)}f(t) = e^{-i\xi t}f(t-x)$.

The following example is motivated by the counterexample on $\Xi = \mathbb Z \times \mathbb T$. Indeed, there we considered a matrix, i.e., an integral operator, which gave the example. We could have written the operator there equally well as the composition of the integral operator and the operator of convolution by $\delta_0 \in \ell^1(\mathbb Z)$ (which is just the identity). The counterexample on $\Xi = \mathbb R \times \mathbb R$ will be exactly of this form: A product of an integral operator (which will again be an orthogonal projection) and a composition operator. In particular, the existence of such an example will prove:
\begin{prop}
    We have $\mathcal K(L^2(\mathbb R)) \subsetneq B_0(L^2(\mathbb R))$.
\end{prop}
\begin{proof}
We consider the integral operator $A$, given by the integral kernel
\begin{align*}
    k(x,y) = \sum_{n=1}^\infty \frac{1}{n} \mathbf 1_{[n^2 - \frac{n}{2}, n^2 + \frac{n}{2}]}(x) \mathbf 1_{[n^2 - \frac{n}{2}, n^2 + \frac{n}{2}]}(y).
\end{align*}
Similarly to the example for $\Xi = \mathbb Z \times \mathbb T$, one verifies that $A$ is an orthogonal projection. Indeed, each operator with integral kernel 
\begin{align*}
k_n(x,y) = \frac{1}{n} \mathbf 1_{[n^2 - \frac{n}{2}, n^2 + \frac{n}{2}]}(x) \mathbf 1_{[n^2 - \frac{n}{2}, n^2 + \frac{n}{2}]}(y)
\end{align*}
is an orthogonal projection, and $A$ is the direct sum of these. In particular, $A$ is bounded. Further, when $(x_\gamma, \xi_\gamma)$ is a net in $\mathbb R^2$ such that $|x_\gamma| \to \infty$, then $\| AW_{(-x_\gamma, -\xi_\gamma)}\varphi\|^2 \to 0$ for every $\varphi \in C_c(\mathbb R)$, i.e., $\alpha_{(x_\gamma, \xi_\gamma)}(A) \to 0$ in SOT$^\ast$ for such nets. Additionally, we consider the operator $C$ of convolution by $\mathbf 1_{[-1, 1]}$. Since $\alpha_{(x, \xi)}(C) = \alpha_{(0, \xi)}(C)$ and $\mathcal F^{-1} \alpha_{(0, \xi)} (C) \mathcal F = M_{\alpha_{\xi}(\widehat{f})}$ with $\widehat{f} \in C_0(\mathbb R) \subset \operatorname{BUC}(\mathbb R)$, it is not hard to verify that $\alpha_{(x, \xi)}(C) \to 0$ in SOT$^\ast$ when $|\xi| \to \infty$. In particular, $\alpha_{(x, \xi)}(CAC) \to 0$ in SOT$^\ast$ when $(x, \xi) \to \partial \Xi$, i.e., $CAC \in B_0(\mathcal H)$. 

It remains to verify that $CAC \not \in \mathcal K(\mathcal H)$, which then shows that $CAC \in B_0(\mathcal H) \setminus \mathcal C_1(\mathcal H)$. Using the formula for the integral kernel of the composition of two integral operators,
\begin{align*}
    k_{AB}(x,y) = \int_{\mathbb R}k_A(x,z)k_B(z,y)~dz,
\end{align*}
one verifies that the integral kernel of $CAC$ is given by
\begin{align*}
    k'(x,y) = \sum_{n=1}^\infty \frac{1}{n} (\mathbf 1_{[n^2 - \frac{n}{2}, n^2 + \frac{n}{2}]} \ast \mathbf 1_{[-1, 1]})(x) \cdot (\mathbf 1_{[n^2 - \frac{n}{2}, n^2 + \frac{n}{2}]} \ast \mathbf 1_{[-1, 1]})(y)
\end{align*}
Elementary computations show that $g(x) = \mathbf 1_{[n^2 - \frac{n}{2}, n^2 + \frac{n}{2}]} \ast \mathbf 1_{[-1, 1]}(x)$ is given by:
\begin{align*}
    g(x) = \begin{cases}
        0, \quad &x \leq n^2 - \frac{n}{2} - 1\\
        x + (1 + \frac{n}{2} - n^2), \quad &n^2 - \frac{n}{2} - 1\leq x \leq n^2 - \frac{n}{2} + 1\\
        2, \quad &n^2 - \frac{n}{2} + 1 \leq x \leq n^2 + \frac n2 - 1\\
        -x + (n^2 + \frac n2 + 1), \quad &n^2 + \frac{n}{2} - 1 \leq x \leq n^2 + \frac{n}{2} + 1\\
        0, \quad &n^2 + \frac n2 + 1 \leq x
    \end{cases}
\end{align*}
In particular,
\begin{align*}
    \mathbf 1_{[n^2 - \frac{n}{2}, n^2 + \frac{n}{2}]} \ast \mathbf 1_{[-1, 1]}(x) \geq \mathbf 1_{[n^2 - \frac{n}{2}, n^2 + \frac{n}{2}]}(x)
\end{align*}
such that $k'(x,y) \geq k(x,y) \geq 0$ for every $x, y \in \mathbb R$.  Then, the integral operator $A_n$ with kernel $k_n$ is an orthogonal projection (with finite rank). The function $f_n = \frac{1}{\sqrt{n}} \mathbf 1_{[n^2 - \frac{n}{2}, n^2 + \frac n2]}$ is of norm one and in the range of $A_n$. Since the projections $A_n$ have orthogonal ranges, we have $f_n \to 0$ weakly as $n \to \infty$, but (as they are normalized) clearly not $f_n \to 0$ in norm. Now, we have\begin{align*}
    CACf_n(x) &= \int_{\mathbb R} k'(x,y)f_n(y)~dy \geq \int_{\mathbb R} k(x,y) f_n(y)~dy \\
    &= Af_n(x) = A_n f_n(x) = f_n(x),
\end{align*}
i.e., $CACf_n \geq f_n$, such that $\| CACf_n\| \geq 1$. Therefore, $CACf_n$ does not converge (in norm) to 0 as $n \to \infty$, hence $CAC$ is not compact.
\end{proof}
Since the operator $CAC$ from the previous proof is not contained in $\mathcal K(\mathcal H)$, but nevertheless all its limit operators are zero, it can also not be contained in $\mathcal C_1(\mathcal H)$. In particular, with $R$ being the parity operator, $RCAC \not \in \mathcal E(\mathcal H)$. Since $\alpha_{(x, \xi)}(CAC)$ converges in SOT to its limit operators (of which we have seen that they are all zero), one sees that the limit operators of $RCAC$ exist in WOT and are all zero. Hence, just as for the setting of $\Xi = \mathbb Z \times \mathbb T$, we see that:
\begin{prop}
    We have $\mathcal E(L^2(\mathbb R)) \subsetneq B_{0, w}(L^2(\mathbb R))$.
\end{prop}

\section{Uniform Wiener Tauberian theorems}\label{sec:uniform}

In this last part of the paper, we are going to discuss uniform versions of Wiener's Tauberian theorem. Even the uniform version of the classical Wiener's Tauberian theorem for functions seems to be not well-known; it seems that it was only discussed rather recently in \cite{Bhatta2004}. There, the statement is discussed and proved for $G = \mathbb R$, but the statement holds for arbitrary locally compact abelian groups. Indeed the proof can be significantly simplified under more general assumptions, simply by making use of compactness arguments in $L^1(G)$. Here is the result under more general assumptions: 
\begin{thm}\label{thm:uniform_wiener_theorem}
    Let $G$ be a locally compact abelian group, $S \subset L^1(G)$ a regular subset and $H \subset L^1(G)$ relatively compact. Then, if $X \subset L^\infty(G)$ is a bounded subset such that for each $g \in S$ we have
    \begin{align*}
        \lim_{x \to \infty} \sup_{f \in X} |g\ast f(x)| = 0,
    \end{align*}
    then it follows that
    \begin{align*}
        \lim_{x \to \infty} \sup_{h \in H} \sup_{f \in X} |h \ast f(x)| = 0.
    \end{align*}
\end{thm}
\begin{proof}
    First of all, note that the assumption easily yields, for each $h \in V := \operatorname{span}\{ \alpha_x(g): ~x \in G, g \in S\}$, that:
    \begin{align*}
        \lim_{x \to \infty} \sup_{f \in X} |h \ast f(x)| = 0.
    \end{align*}
    Let $\varepsilon > 0$. Then, by relative compactness, $H$ can be covered by $n \in \mathbb N$ balls $B(h_j, \varepsilon)$, with $h_j \in V$. Therefore, for each $h \in H$, we can find $h_j \in V$ with $\| h - h_j\|_{L^1} < \varepsilon$. Hence, for each $h \in H$ and $f \in X$ we have (where we write $C = \sup_{f \in X} \| f\|_\infty < \infty$):
    \begin{align*}
        |h \ast f(x)| &\leq |h_j \ast f(x)| + |(h - h_j) \ast f(x)|\\
        &\leq |h_j \ast f(x)| + \| h - h_j\|_{L^1} \| f\|_\infty\\
        &\leq |h_j \ast f(x)| + \varepsilon C.
    \end{align*}
    In particular, 
    \begin{align*}
        \sup_{h \in H} \sup_{f \in X} |h \ast f(x)| \leq \sup_{j = 1, \dots, n} \sup_{f \in X} |h_j \ast f(x)| + \varepsilon C.
    \end{align*}
    Since for each $j = 1, \dots, n$, we know that $\sup_{f \in X} |h_j \ast f(x)| \to 0$, we see that:
    \begin{align*}
        \limsup_{x \to \infty} \sup_{h \in H} \sup_{f \in X} |h \ast f(x)| \leq \varepsilon C.
    \end{align*}
    Since $\varepsilon > 0$ was arbitrary, the statement follows.
\end{proof}
\begin{rem}
    We recall the theorem of Riesz and Kolmogorov, see \cite{Sudakov,Hanche}, characterizing relatively compact subsets in $L^1(G)$ as subsets $H \subset L^1(G)$ satisfying 
    \begin{enumerate}
        \item $H$ is uniformly equicontinuous, i.e., $\sup_{h \in H} \| \alpha_x(h) - h\|_{L^1} \to 0$ as $x \to 0$,
        \item For each $\varepsilon > 0$ there exists a compact $K \subset G$ such that
        \begin{align*}
            \sup_{h \in H} \int_{K^c} |h(x)|~dx < \varepsilon.
        \end{align*}
    \end{enumerate}
    Using this characterization, it is not hard to prove that the above Theorem \ref{thm:uniform_wiener_theorem} indeed implies the theorem in \cite{Bhatta2004}.
\end{rem}
The uniform Wiener Tauberian theorem above can be easily reformulated, with essentially the same proof, in terms of limit functions:
\begin{cor}
    Let $G$ be a locally compact abelian group, $S \subset L^1(G)$ a regular subset and $H \subset L^1(G)$ relatively compact. Then, if $X \subset L^\infty(G)$ is a bounded subset, $x \in \partial G$ and $(x_\gamma)_{\gamma \in \Gamma} \subset G$ a net converging to $x$ such that for each $g \in S$ we have
    \begin{align*}
        \sup_{f \in X} |g \ast [\alpha_{x_\gamma}(f) - \alpha_x(f)]| \overset{\gamma \in\Gamma}{\longrightarrow} 0
    \end{align*}
    uniformly on compact subsets of $G$, then
    \begin{align*}
        \sup_{h \in H} \sup_{f \in X} |h \ast [\alpha_{x_\gamma}(f) - \alpha_x(f)]| \overset{\gamma \in \Gamma}{\longrightarrow} 0
    \end{align*}
    uniformly on compact subsets of $G$.
\end{cor}
Having obtained the reformulation in terms of limit functions, it is now not hard to obtain (by essentially the same reasoning) the following uniform versions of Tauberian theorems within the setting of quantum harmonic analysis:
\begin{thm}
    Let $\Xi$ be a phase space, $S \subset L^1(\Xi)$ a regular subset and $H \subset L^1(\Xi)$ relatively compact. If $X \subset \mathcal L(\mathcal H)$ is a bounded subset, $x \in \partial \Xi$ and $(x_\gamma)_{\gamma \in \Gamma} \subset \Xi$ a net converging to $x$ such that for each $g \in S$ and $\varphi \in \mathcal H$ we have
    \begin{align*}
        \sup_{B \in X} \| g \ast [\alpha_{x_\gamma}(B) - \alpha_x(B)](\varphi) \| \overset{\gamma \in \Gamma}{\longrightarrow} 0,
    \end{align*}
    then also
    \begin{align*}
        \sup_{h \in H} \sup_{B \in X} \| h \ast [\alpha_{x_\gamma}(B) - \alpha_x(B)](\varphi) \| \overset{\gamma \in \Gamma}{\longrightarrow} 0 .
    \end{align*}
\end{thm}
\begin{thm}
    Let $\Xi$ be a phase space, $S \subset \mathcal T^1(\mathcal H)$ a regular subset and $H\subset \mathcal T^1(\mathcal H)$ relatively compact. If $X \subset L^\infty(\Xi)$ is a bounded subset, $x \in \partial \Xi$ and $(x_\gamma)_{\gamma \in \Gamma} \subset \Xi$ a net converging to $x$ such that for each $A_0 \in S$ and $\varphi \in \mathcal H$ we have
    \begin{align*}
        \sup_{f \in X} \| A_0 \ast [\alpha_{x_\gamma}(f) - \alpha_x(f)](\varphi) \| \overset{\gamma \in \Gamma}{\longrightarrow} 0,
    \end{align*}
    then also
    \begin{align*}
        \sup_{A \in H} \sup_{B \in X} \| A \ast [\alpha_{x_\gamma}(f) - \alpha_x(f)](\varphi) \| \overset{\gamma \in \Gamma}{\longrightarrow} 0 .
    \end{align*}
\end{thm}
\begin{thm}\label{thm:unif_tauberian_4}
    Let $\Xi$ be a phase space, $S \subset \mathcal T^1(\mathcal H)$ a regular subset and $H\subset \mathcal T^1(\mathcal H)$ relatively compact. If $X \subset \mathcal L(\mathcal H)$ is a bounded subset, $x \in \partial \Xi$ and $(x_\gamma)_{\gamma \in \Gamma} \subset \Xi$ a net converging to $x$ such that for each $A_0 \in S$ we have
    \begin{align*}
        \sup_{B \in X} | A_0 \ast [\alpha_{x_\gamma}(B) - \alpha_x(B)]| \overset{\gamma \in \Gamma}{\longrightarrow} 0
    \end{align*}
    uniformly on compact subsets of $\Xi$, then also
    \begin{align*}
        \sup_{A \in H} \sup_{B \in X} | A \ast [\alpha_{x_\gamma}(B) - \alpha_x(B)]| \overset{\gamma \in \Gamma}{\longrightarrow} 0 
    \end{align*}
    uniformly on compact subsets of $\Xi$. 
\end{thm}
We also mention the following version of the previous theorem, which is more in the spirit of Theorem \ref{thm:uniform_wiener_theorem} and proven analogously.
\begin{thm}\label{thm:unif_tauberian_5}
    Let $\Xi$ be a phase space, $S \subset \mathcal T^1(\mathcal H)$ a regular subset and $H\subset \mathcal T^1(\mathcal H)$ relatively compact. If $X \subset \mathcal L(\mathcal H)$ is a bounded subset such that for each $A_0 \in S$ we have
    \begin{align*}
        \sup_{B \in X} | A_0 \ast B(x)| \overset{x \to \infty}{\longrightarrow} 0
    \end{align*}
    uniformly on compact subsets of $\Xi$, then also
    \begin{align*}
        \sup_{A \in H} \sup_{B \in X} | A \ast B(x)| \overset{x \to \infty}{\longrightarrow} 0 
    \end{align*}
    uniformly on compact subsets of $\Xi$. 
\end{thm}
We will now discuss some applications of the uniform Tauberian theorems, at least the first of which initially motivated us to consider these results.

We recall that a \emph{localization operator} (in the sense of time-frequency analysis) is, simply speaking, a convolution $\mathcal A_f^{\varphi, \psi} := f \ast (\varphi \otimes \psi)$. The following result is obtained as a combination of the results in \cite{Fernandez_Galbis2006} and \cite[Proposition 4.3]{Luef_Skrettingland2021}.
\begin{thm}
    Let $f \in L^\infty(\mathbb R^{2d})$. Then, the following are equivalent:
    \begin{enumerate}[(1)]
        \item $\mathcal A_f^{\varphi, \varphi}$ is compact for every $\varphi \in L^2(\mathbb R^d)$. 
        \item There exists a Schwartz function $0 \neq \Phi \in \mathcal S(\mathbb R^{2d})$ such that for every $R > 0$: \begin{align*}
            \lim_{|x| \to \infty} \sup_{|\omega| \leq R} |V_\Phi f(x, \omega)| = 0.
        \end{align*}
        \item $f \ast S$ is compact for every $S \in \mathcal T^1(\mathcal H)$.
        \item There exists some regular $a \in L^1(\mathbb R^{2d})$ such that
        \begin{align*}
            \lim_{|x| \to \infty} f \ast a(x) = 0.
        \end{align*}
    \end{enumerate}
\end{thm}
Here, we have used the notion $V_\Phi f(x, \omega) = \langle f, \pi(x, \omega) \Phi\rangle$ for the short time Fourier transform, where $\pi(x,\omega)\Phi(t) = e^{i\omega t} \Phi(x-t)$. 

Indeed, the equivalence of (2) and (4) in the above theorem could only be proven, in the above-given references, by passing through operators. This approach has the clear drawback that it only works on even-dimensional vector spaces (or, more generally, abelian phase spaces). The uniform version of Wiener's Tauberian theorem allows for a proof of this equivalence on arbitrary locally compact abelian groups, i.e., the following result. Here, we used a slight renormalization of the short time Fourier transform, which differs only by some constants from the convention in \cite{Luef_Skrettingland2021} for the case $G = \mathbb R^d$.
\begin{thm}\label{thm:appl1}
    Let $G$ be a locally compact abelian group and $f \in L^\infty(G)$. Then, the following are equivalent:
    \begin{enumerate}[(1)]
        \item There exists some $0 \neq \Phi \in L^1(G)$ such that, for the STFT $V_\Phi f(x, \omega)$ defined as
        \begin{align*}
            V_\Phi f(x, \xi) := \int_G \Phi(t) \xi(t) f(t-x)~dt, \quad x \in G, \xi \in \widehat{G},
        \end{align*}
        the following holds true: For each compact subset $K \subset \widehat{G}$  we have:
        \begin{align*}
            \lim_{x \to \infty} \sup_{\xi \in K} |V_\Phi f(x, \xi)| = 0.
        \end{align*}
        \item There exists a regular subset $S \subset L^1(G)$ such that for each $g \in S$:
        \begin{align*}
            \lim_{x \to\infty} g \ast f(x) = 0.
        \end{align*}
    \end{enumerate}
\end{thm}
\begin{proof}
    The proof hinges on the simple fact that $V_\Phi f(x, \xi) = [\Phi \xi] \ast f$. 

    (1) $\Rightarrow$ (2): Note that $\mathcal F(\Phi \xi)(\eta) = \mathcal F(\Phi)(\eta \xi^{-1})$. For $\Phi \neq 0$, there exists some $\eta_0 \in \widehat{G}$ such that $\mathcal F(\Phi)(\eta_0) \neq 0$. In particular, for each $\eta \in \widehat{G}$, we have $\mathcal F(\Phi \eta \eta_0^{-1})(\eta) = \mathcal F(\Phi)(\eta_0) \neq 0$. Hence, $\{ \Phi \eta: ~\eta \in \widehat{G}\}$ is a regular family in $L^1(G)$. 

    The assumption implies now, in particular, that for each $\xi \in \widehat{G}$: $[\Phi \xi] \ast f \in C_0(G)$. Hence, we arrive at the statement of (2).

    (2) $\Rightarrow$ (1). Let $0 \neq \Phi \in L^1(G)$ be arbitrary. By the classical Wiener Tauberian theorem, we see that for each $\xi \in \widehat{G}$: $\lim_{x \to \infty} V_\Phi f(x, \xi) = \lim_{x \to \infty} (\Phi \xi) \ast f(x) = 0$. Hence, we only have to prove that this convergence is uniformly for $\xi$ chosen from a compact subset $K$ of $\widehat{G}$. For doing so, we note that the map $\widehat{G} \ni \xi \mapsto \Phi \xi \in L^1(G)$ is continuous, as one readily verifies. Since continuous images of compact spaces are compact, we obtain that for every compact subset $K \subset \Xi$, the set $\{ \Phi \xi: ~\xi \in K\}$ is compact in $L^1(G)$. Hence, Theorem \ref{thm:uniform_wiener_theorem} proves (with $X = \{ f\}$):
    \begin{align*}
        \sup_{\xi \in K} |(\Phi \xi) \ast f(x)| = \sup_{\xi \in K} |V_\Phi f(x, \xi)| \overset{x \to \infty}{\longrightarrow} 0.
    \end{align*}
    This finishes the proof.
\end{proof}
Having the QHA-analogues of the uniform Tauberian theorems at hand, applications to operator theory are imminent. We end this discussion by giving a nice compactness criterion, which does not hinge on the necessity of finding a regular family $S \subset \mathcal T^1(\mathcal H)$.
\begin{thm}
    Let $B \in \mathcal L(\mathcal H)$. Then, the following are equivalent:
    \begin{enumerate}[(1)]
        \item $B \in \mathcal K(\mathcal H)$.
        \item $B \in \mathcal C_1(\mathcal H)$ and for some $0 \neq A \in \mathcal T^1(\mathcal H)$ we have for every compact $K \subset \Xi$:
        \begin{align*}
            \sup_{x \in K} |(U_x A) \ast B(y)| \overset{y \to \infty}{\longrightarrow} 0.
        \end{align*}
    \end{enumerate}
\end{thm}
\begin{proof}
    By Theorem \ref{thm:standardcorrespondences}, we know that $B\in \mathcal K(\mathcal H)$ if and only if $B \in \mathcal C_1(\mathcal H)$ and for any regular subset $S \subset \mathcal T^1(\mathcal H)$ we have $A \ast B \in C_0(\Xi)$ for every $A\in S$. Now, if $0 \neq A \in \mathcal T^1(\mathcal H)$, it is not hard to show that $\{ U_y A: ~y \in \Xi\}$ is a regular subset (this is analogous to the argument in the proof of (1) $\Rightarrow$ (2) of Theorem \ref{thm:appl1}). Hence, (2) $\Rightarrow$ (1) follows. To derive the implication (1) $\Rightarrow$ (2), we need to verify that for each compact subset $K \subset \Xi$, $\{ U_y A: ~x \in K\}$ is relatively compact in $\mathcal T^1(\mathcal H)$. Since for each $A \in \mathcal T^1(\mathcal H)$, the map $\Xi \ni y \mapsto U_y A \in \mathcal T^1(\mathcal H)$ is continuous (this is readily verified for finite rank operators), and the continuous image of a compact set is compact, this fact is obtained. Hence, Theorem \ref{thm:unif_tauberian_5} (applied with $X = B$) shows:
    \begin{align*}
        \sup_{y \in K} |(U_y A) \ast B(x)| \overset{x \to \infty}{\longrightarrow} 0.
    \end{align*}
   This concludes the proof. 
\end{proof}

\subsection*{Acknowledgments:} We acknowledge the anonymous referee's detailed comments on the paper, which helped improving the presentation significantly.

\printnomenclature

\bibliographystyle{amsplain}
\bibliography{References}

\providecommand{\bysame}{\leavevmode\hbox to3em{\hrulefill}\thinspace}
\providecommand{\MR}{\relax\ifhmode\unskip\space\fi MR }
\providecommand{\MRhref}[2]{%
  \href{http://www.ams.org/mathscinet-getitem?mr=#1}{#2}
}
\providecommand{\href}[2]{#2}
\begin{thebibliography}{10}

\bibitem{Berge_Berge_Luef_Skrettingland2022}
E.~Berge, S.~M. Berge, F.~Luef, and E.~Skrettingland, \emph{{Affine quantum harmonic analysis}}, J. Funct. Anal. \textbf{282} (2022), 109327.

\bibitem{Bhatta2004}
C.~R. Bhatta, \emph{{Uniform version of the Wiener-Tauberian theorem for real line}}, Nepali Math. Sci. Rep. \textbf{23} (2004), 9--15.

\bibitem{Bourbaki1987}
N.~Bourbaki, \emph{{General Topology, Chapters 1-4}}, Springer Verlag, 1987.

\bibitem{Cornulier_Harpe2016}
Y.~Cornulier and P.~de~la Harpe, \emph{{Metric Geometry of Locally Compact Groups}}, EMS Tracts in Mathematics, vol.~25, European Mathematical Society, Z\"{u}rich, 2016.

\bibitem{Dammeier_Werner2023}
L.~Dammeier and R.~F. Werner, \emph{Quantum-classical hybrid systems and their quasifree transformations}, Quantum \textbf{7} (2023), 1068.

\bibitem{Feichtinger81}
H.~G. Feichtinger, \emph{On a new {S}egal algebra}, Monatsh. Math. \textbf{92} (1981), no.~4, 269--289.

\bibitem{Fernandez_Galbis2006}
C.~Fernand\'{e}z and A.~Galbis, \emph{{Compactness of time-frequency localization operators on $L^2(\mathbb R^d)$}}, J.~ Funct.~ Anal. \textbf{233} (2006), 335--350.

\bibitem{Folland2016}
G.~B. Folland, \emph{A course in abstract harmonic analysis}, 2nd ed., CRC Press, 2016.

\bibitem{Fulsche2020}
R.~Fulsche, \emph{{Correspondence theory on $p$-Fock spaces with applications to Toeplitz algebras}}, J. Funct. Anal. \textbf{279} (2020), 108661.

\bibitem{Fulsche2024}
\bysame, \emph{Toeplitz operators on non-reflexive {F}ock spaces}, Rev. Mat. Iberoam. \textbf{40} (2024), 1115–1148.

\bibitem{Fulsche_Galke2023}
R.~Fulsche and N.~Galke, \emph{Quantum harmonic analysis on locally compact abelian groups}, J. Fourier Anal. Appl. \textbf{31} (2025), article number 13.

\bibitem{Fulsche_Hagger2023}
R.~Fulsche and R.~Hagger, \emph{{Quantum harmonic analysis for polyanalytic Fock spaces}}, J. Fourier Anal. Appl. \textbf{30} (2024), article number 63.

\bibitem{Fulsche_Rodriguez2023}
R.~Fulsche and M.~A. Rodriguez~Rodriguez, \emph{Commutative {{\(G\)}}-invariant {Toeplitz} {{\({C}^{{\ast}} \)}}-algebras on the {Fock} space and their {Gelfand} theory through quantum harmonic analysis}, J. Oper. Theory \textbf{93} (2025), no.~2, 593--620.

\bibitem{Halmos1982}
P.~R. Halmos, \emph{{A Hilbert Space Problem Book }}, 2nd ed., Graduate Texts in Mathematics, Springer, 1982.

\bibitem{Halvdansson2022}
S.~Halvdansson, \emph{Quantum harmonic analysis on locally compact groups}, J. Funct. Anal. \textbf{285} (2023), 110096.

\bibitem{Hanche}
H.~Hanche-Olsen, H.~Holden, and E.~Malinnikova, \emph{An improvement of the {Kolmogorov}-{Riesz} compactness theorem}, Expo. Math. \textbf{37} (2019), no.~1, 84--91.

\bibitem{hewitt_ross_2}
E.~Hewitt and K.~Ross, \emph{{Abstract Harmonic Analysis, Volume II}}, Springer Verlag, 1970.

\bibitem{jakobsen18}
M.~S. Jakobsen, \emph{{On a (no longer) new Segal algebra: A review of the Feichtinger algebra}}, J. Fourier Anal. Appl. \textbf{24} (2018), 1579--1660.

\bibitem{keyl_kiukas_werner16}
M.~Keyl, J.~Kiukas, and R.~Werner, \emph{Schwartz operators}, Rev. Math. Phys. \textbf{28} (2016), 1630001.

\bibitem{Korevaar2004}
J.~Korevaar, \emph{{Tauberian Theory}}, {Grundlehren der mathematischen Wissenschaften}, vol. 329, Springer, 2004.

\bibitem{Luef_Skrettingland2018a}
F.~Luef and E.~Skrettingland, \emph{Convolutions for localization operators}, J. Math. Pures Appl. \textbf{118} (2018), 288--316.

\bibitem{Luef_Skrettingland2019a}
\bysame, \emph{{Mixed-state localization operators: Cohen's class and trace class operators}}, J. Fourier Anal. Appl. \textbf{25} (2019), 2064–2108.

\bibitem{Luef_Skrettingland2021}
\bysame, \emph{A {W}iener {T}auberian theorem for operators and functions}, J. Funct. Anal. \textbf{280} (2021), 108883.

\bibitem{reiter20}
H.~Reiter and J.~D. Stegeman, \emph{{Classical Harmonic Analysis and Locally Compact Groups}}, London Mathematical Society Monographs New Series, vol.~22, Oxford University Press, 2000.

\bibitem{Sudakov}
V.~N. Sudakov, \emph{Zur {Frage} der {Kompaktheitskriterien} in {Funktionenr{\"a}umen}}, Usp. Mat. Nauk \textbf{12} (1957), no.~3(75), 221--224.

\bibitem{werner84}
R.~Werner, \emph{{Quantum Harmonic Analysis on Phase Space}}, J. Math. Phys. \textbf{25} (1984), 1404–1411.

\bibitem{Wiener1932}
N.~Wiener, \emph{Tauberian theorems}, Ann. of Math. (2) \textbf{33} (1932), 1–100.

\end{thebibliography}
\vspace{1cm}
\begin{multicols}{2}
\noindent
Robert Fulsche\\
\href{fulsche@math.uni-hannover.de}{\Letter ~fulsche@math.uni-hannover.de}
\\
\noindent
Institut f\"{u}r Analysis\\
Leibniz Universit\"at Hannover\\
Welfengarten 1\\
30167 Hannover\\
GERMANY\\ 

\noindent
Franz Luef\\
\href{franz.luef@ntnu.no}{\Letter ~franz.luef@ntnu.no}
\\
\noindent
Department of Mathematical Sciences\\
NTNU Trondheim\\
Alfred Getz vei 1\\
7034 Trondheim\\
NORWAY
\end{multicols}

\noindent Reinhard F.\ Werner\\
\href{reinhard.werner@itp.uni-hannover.de}{\Letter ~reinhard.werner@itp.uni-hannover.de}
\\
\noindent
Institut f\"ur Theoretische Physik\\
Leibniz Universit\"at Hannover\\
Schneiderberg 32\\
30167 Hannover\\
GERMANY\\

\end{document}